\documentclass{article}%
\usepackage{amsmath}
\usepackage{amsfonts}
\usepackage{amssymb}
\usepackage{graphicx}
\usepackage{hyperref}%
\newtheorem{theorem}{Theorem}

\newtheorem{corollary}[theorem]{Corollary}

\newtheorem{definition}[theorem]{Definition}
\newtheorem{example}[theorem]{Example}

\newtheorem{lemma}[theorem]{Lemma}

\newtheorem{proposition}[theorem]{Proposition}
\newtheorem{remark}[theorem]{Remark}

\newenvironment{proof}[1][Proof]{\noindent\textbf{#1.} }{\ \rule{0.5em}{0.5em}}
\begin{document}

\title{Controllability properties and invariance pressure for linear discrete-time systems}
\author{Fritz Colonius\\Institut f\"{u}r Mathematik, Universit\"{a}t Augsburg, Augsburg, Germany
\and Jo\~{a}o A. N. Cossich and Alexandre J. Santana\\Departamento de Matem\'{a}tica, Universidade Estadual de Maring\'{a}\\Maring\'{a}, Brazil}
\maketitle

\textbf{Abstract}\footnote{We have announced some results of the present paper
in \textquotedblleft Invariance pressure for linear discrete-time
systems\textquotedblright, Proceedings of the 2019 IEEE Information Theory
Workshop (IEEE ITW 2019), Visby, Sweden, 24-26 Aug. 2019.}\textbf{. }For
linear control systems in discrete time controllability properties are
characterized. In particular, a unique control set with nonvoid interior
exists and it is bounded in the hyperbolic case. Then a formula for the
invariance pressure of this control set is proved.


\textbf{MSC\ 2010.} 93B05, 37B40, 94A17

\section{Introduction}

Invariance pressure for subsets of the state space generalizes invariance
entropy of deterministic control systems by adding potentials on the control
range. We consider control systems in discrete time of the form%
\[
x_{k+1}=F(x_{k},u_{k}),k\in\mathbb{N}_{0}=\{0,1,\ldots\},
\]
where $F:M\times U\rightarrow M$ is smooth for a smooth manifold $M$ and a
compact control range $U\subset\mathbb{R}^{m}$. The invariance entropy
$h_{inv}(K,Q)$ determines the average data rate needed to keep the system in
$Q$ (forward in time) when it starts in $K\subset Q$. Basic references for
invariance entropy are Nair, Evans, Mareels, and Moran \cite{NEMM04} and the
monograph Kawan \cite{Kawa13}, where also the relation to minimal data rates
is explained. With some analogy to classical constructions for dynamical
systems, invariance pressure adds continuous functions $f:U\rightarrow
\mathbb{R}$ called potentials giving a weight to the control values.

For continuous-time systems, invariance entropy of hyperbolic control sets has
been analyzed in Kawan \cite{Kawa11b} and Kawan and Da Silva \cite{KawaDS16}.
Kawan and Da Silva \cite{KawaDS18} and \cite{KawaDS19} analyze invariance
entropy of partially hyperbolic controlled invariant sets and chain control
sets. Huang and Zhong \cite{HuanZ18} show dimension-like characterizations of
invariance entropy. Measure-theoretic versions of invariance entropy have been
considered in Colonius \cite{Colo18} and Wang, Huang, and Sun \cite{WangHS19}.
Invariance pressure has been analyzed in Colonius, Cossich, and Santana
\cite{Cocosa1, Cocosa2, Cocosa3}. In Zhong and Huang \cite{ZHuag19} it is
shown that several generalized notions of invariance pressure fit into the
dimension-theoretic framework due to Pesin.

The main results of the present paper are given for linear control systems
$x_{k+1}=Ax_{k}+Bu_{k}$ with an invertible matrix $A$ and control values
$u_{k}$ in a compact neighborhood $U$ of the origin in $\mathbb{R}^{m}$. It is
shown that a unique control set $D$ with nonvoid interior exists if and only
if the system without control constraints is controllable (i.e., the pair
$(A,B)$ is controllable), and $D$ is bounded if and only if $A$ is hyperbolic.
In this case a formula for the invariance pressure of compact subsets $K$ in
$D$ is presented.

The contents of this paper are as follows: Section \ref{section2} collects
general properties of control sets for nonlinear discrete-time systems.
Section \ref{section3} characterizes controllability properties of linear
discrete-time systems with control constraints and Section \ref{section4}
shows that here a unique control set with nonvoid interior exists and that it
is bounded if and only if the uncontrolled system is hyperbolic. Section
\ref{section5} introduces invariance entropy and as a generalization total
invariance pressure where potentials on the product of the state space and the
control range are allowed. For linear systems, Section \ref{section6} first
derives an upper bound for the total invariance pressure and a lower bound for
the invariance pressure. Combined they yield a formula for the invariance
pressure in the hyperbolic case.

\section{Control sets for nonlinear systems\label{section2}}

In this section we introduce some notation and prove several properties of
control sets with nonvoid interior for nonlinear discrete-time systems. They
are analogous to properties of systems in continuous time, however, the
statements are a bit more involved, since one has to consider in addition to
the interior of control sets their transitivity sets. A discussion of various
slightly differing versions in the literature is contained in Colonius
\cite[Section 5]{Colo18}.

We consider control systems of the form%
\begin{equation}
x_{k+1}=F(x_{k},u_{k}),k\in\mathbb{N}_{0},\label{nonlinear}%
\end{equation}
on a $C^{\infty}$-manifold $M$ of dimension $d$ endowed with a corresponding
metric. For an initial value $x_{0}\in M$ at time $k=0$ and control
$u=(u_{k})_{k\geq0}\in\mathcal{U}:=U^{\mathbb{N}_{0}}$ we denote the solutions
by $\varphi(k,x_{0},u),k\in\mathbb{N}_{0}$. Assume that the set of control
values $U\subset\mathbb{R}^{m}$ is nonvoid and satisfies $U\subset
\overline{\mathrm{int}U}$. Let $\tilde{U}$ be an open set containing
$\overline{U}$ and suppose that the map $F:M\times\tilde{U}\rightarrow M$ is a
$C^{\infty}$-map.

\begin{definition}
For $x\in M$ and $k\in\mathbb{N}$ the reachable set $\mathbf{R}_{k}(x)$ and
the controllable set $\mathbf{C}_{k}(x)$ are%
\begin{align*}
\mathbf{R}_{k}(x)  &  :=\{y\in M\left\vert \exists u\in\mathcal{U}%
:y=\varphi(k,x,u)\right.  \},\\
\mathbf{C}_{k}(x)  &  :=\{y\in M\left\vert \exists u\in\mathcal{U}%
:\varphi(k,y,u)=x\right.  \},
\end{align*}
resp., and $\mathbf{R}(x)$ and $\mathbf{C}(x)$ are the respective unions over
all $k\in\mathbb{N}$. The system is called accessible in $x$ if%
\begin{equation}
\mathrm{int}\mathbf{R}(x)\not =\varnothing\text{ and }\mathrm{int}%
\mathbf{C}(x)\not =\varnothing. \label{access0}%
\end{equation}

\end{definition}

Accessibility in $x$ certainly holds if%
\[
\mathrm{int}F(x,U)\not =\varnothing\text{ and }\mathrm{int}\{y\in M\left\vert
x\in F(y,U)\right.  \}\not =\varnothing.
\]

Next we specify maximal subsets of complete approximate controllability.

\begin{definition}
\label{Definition3.1}For system of the form (\ref{nonlinear}) a nonvoid subset
$D\subset M$ is called a control set if it is maximal with (i) $D\subset
\overline{\mathbf{R}(x)}$ for all $x\in D$, (ii) for every $x\in D$ there is
$u\in\mathcal{U}$ with $\varphi(k,x,u)\in D$ for all $k\in\mathbb{N}$. The
transitivity set $D_{0}$ of $D$ is $D_{0}:=\{z\in D\left\vert z\in
\mathrm{int}\mathbf{C}(z)\right.  \}$.
\end{definition}

We define for $k\geq1$ a $C^{\infty}$-map%
\[
G_{k}:M\times U^{k}\rightarrow M,G_{k}(x,u):=\varphi(k,x,u).
\]
Following Wirth \cite{Wirth98} we say that a pair $(x,u)\in M\times
\mathrm{int}U^{k}$ is regular if $\mathrm{rank}\frac{\partial G_{k}}{\partial
u}(x,u)=d$ (clearly, this implies $mk\geq d$). For $x\in M$ and $k\in
\mathbb{N}$ the regular reachable set and the regular controllable set at time
$k$ are
\begin{align*}
\mathbf{\hat{R}}_{k}(x)  &  :=\left\{  \varphi(k,x,u)\left\vert (x,u)\text{ is
regular}\right.  \right\}  ,\\
\mathbf{\hat{C}}_{k}(x)  &  :=\left\{  y\in M\left\vert x=\varphi(k,y,u)\text{
with }(y,u)\text{ regular}\right.  \right\}  ,
\end{align*}
resp., and the regular reachable set $\mathbf{\hat{R}}(x)$ and controllable
set $\mathbf{\hat{C}}(x)$ are given by the respective union over all
$k\in\mathbb{N}$. It is clear that $\mathbf{\hat{R}}(x)$ and $\mathbf{\hat{C}%
}(x)$ are open for every $x$.

Accessibility condition (\ref{access0}) implies that there is $k_{0}%
\in\mathbb{N}$ such that for all $k\geq k_{0}$ one has $\mathrm{int}%
\mathbf{R}_{k}(x)\not =\varnothing$ and
\[
\mathbf{R}_{k}(x)\subset\overline{\{\varphi(k,x,u)\in\mathrm{int}%
\mathbf{R}_{k}(x)\left\vert u\in\mathrm{int}U^{k}\right.  \}}.
\]
By Sard's Theorem the set of points $\varphi(k,x,u)\in\mathbf{R}_{k}(x)$ such
that $(x,u)$ is not regular has Lebesgue measure zero.

\begin{proposition}
\label{proposition_transitivity}Assume that accessibility condition
(\ref{access0}) holds for all $x\in M$. Then for every control set $D$ with
nonvoid interior the transitivity set $D_{0}$ is nonvoid and dense in
$\mathrm{int}D$.
\end{proposition}

\begin{proof}
For $x\in\mathrm{int}D$ there is $k_{0}\in\mathbb{N}$ such that the reachable
set $\mathbf{R}_{k}(x)$ at time $k$ has nonvoid interior for all $k\geq k_{0}%
$. There is $k\geq k_{0}$ with $\mathbf{R}_{k}(x)\cap\mathrm{int}%
D\not =\varnothing$, hence we may assume that there is $y:=\varphi
(k,x,u)\in\mathrm{int}\mathbf{R}_{k}(x)\cap\mathrm{int}D$. Then, by Sard's
Theorem, it follows that there is a point $y=\varphi(k,x,u)\in\mathrm{int}D$
with some regular $(x,u)$, i.e., $y\in\mathrm{int}D\cap\mathbf{\hat{R}}%
_{k}(x)$. Then $x\in\mathrm{int}\mathbf{C}(y)$. Let $V\subset\mathrm{int}%
\mathbf{C}(y)$ be a neighborhood of $x$. Since $x\in\mathrm{int}D$ and
$D\subset\overline{\mathbf{R}(y)}$, there is $z\in V\cap\mathbf{R}(y)\subset
D$ and thus $y\in\mathbf{C}(z)$. By construction, the point $z\in D$ satisfies
$z\in\mathrm{int}\mathbf{C}(y)\subset\mathrm{int}\mathbf{C}(z)$, hence it is
in the transitivity set of $D$ and $D_{0}$ is dense in $\mathrm{int}D$.
\end{proof}

\begin{remark}
In the general context of semigroups of continuous maps (and with slightly
different notation), Patr\~{a}o and San Martin \cite[ Propositions 4.8 and
4.10]{PatSM07} show that the transitivity set $D_{0}$ is dense in a control
set $D$ with nonvoid interior provided that $D_{0}\not =\varnothing$.
\end{remark}

We note the following further results for control sets.

\begin{proposition}
\label{lemma1} Assume that $D$ is a control set for a control system which is
accessible for all $x\in M$. Then its transitivity set $D_{0}$ satisfies
$D_{0}\subset\mathbf{R}(x)$ for all $x\in D$.
\end{proposition}

\begin{proof}
Let $x\in D$ and $x_{0}\in D_{0}$. By approximate controllability of $D$ and
$x_{0}\in\mathrm{int}\mathbf{C}(x_{0})$, there are $k\geq1$ and $u\in
\mathcal{U}$ with $\varphi(k,x,u)\in\mathrm{int}\mathbf{C}(x_{0})$. Hence
there are $l\geq1$ and $v\in\mathcal{U}$ such that $\varphi(l,\varphi
(k,x,u),v)=x_{0}$. Therefore $x\in\mathbf{C}(x_{0})$, that is, $x_{0}%
\in\mathbf{R}(x)$.
\end{proof}

\begin{proposition}
\label{lem1.1.2} Assume that $D$ is a control set with nonvoid interior of a
control system, which is accessible for all $x\in M$. Then the transitivity
set $D_{0}$ of $D$ is nonvoid and%
\[
D=\overline{\mathbf{R}(x_{0})}\cap\mathbf{C}(x_{0})\text{ for all }x_{0}\in
D_{0},
\]
in particular, the set $D$ is measurable.
\end{proposition}

\begin{proof}
By Proposition \ref{proposition_transitivity} the transitivity set $D_{0}$ is
nonvoid. Let $x_{0}\in D_{0}$. Note that $D\subset\overline{\mathbf{R}(x_{0}%
)}$ by the definition of control sets. For every $x\in D$, Proposition
\ref{lemma1} shows that $x_{0}\in\mathbf{R}(x)$, that is $x\in\mathbf{C}%
(x_{0})$. Hence $D\subset D^{\prime}:=\overline{\mathbf{R}(x_{0})}%
\cap\mathbf{C}(x_{0})$. It is not difficult to see that the set $D^{\prime}$
is a set of approximate controllability with nonvoid interior. It follows that
$D^{\prime}$ is contained in a maximal set $D^{\prime\prime}$ of approximate
controllability with nonvoid interior, which by Kawan \cite[Proposition
1.20]{Kawa13} is a control set. By the maximality property of control sets and
$D\subset D^{\prime\prime}$, it follows that $D=D^{\prime}=D^{\prime\prime}$,
which concludes the proof.
\end{proof}

The following proposition shows that a trajectory starting in the interior of
a control set $D$ and remaining in it up to a positive time must actually
remain in the interior of $D$.

\begin{proposition}
\label{proposition_in}Assume that the maps $F(\cdot,u)$ are local
diffeomorphisms on $M$ for all $u\in U$. Let $x$ be in the interior of a
control set $D$ and suppose that for some $\tau\in\mathbb{N}$ and
$u\in\mathcal{U}$ one has $\varphi(k,x,u)\in D,k\in\{1,\dotsc,\tau\}$. Then
$\varphi(k,x,u)\in\mathrm{int}D,k\in\{1,\dotsc,\tau\}$.
\end{proposition}

\begin{proof}
Suppose that $y:=\varphi(k,x,u)\in D\cap\partial D$ for some $k\in
\{1,\dotsc,\tau\}$. By the assumption on the maps $F(\cdot,u)$ and
$x\in\mathrm{int}D$, there is a neighborhood $N_{0}(y)$ of $y$ with
$N_{0}(y)=\varphi(k,N(x),u)$ for a neighborhood $N(x)\subset D$ of $x$. Since
$y\in D$, there are a control $v\in\mathcal{U}$ and $k_{0}\in\mathbb{N}$ with
$\varphi(k_{0},y,v)\in\mathrm{int}D$. Then there is a neighborhood $N_{1}(y)$
with $\varphi(k_{0},N_{1}(y),u)\subset\mathrm{int}D$. By the maximality
property of control sets it follows that the neighborhood $N_{0}(y)\cap
N_{1}(y)$ of $y$ is contained in $D$, contradicting $y\in\partial D$.
\end{proof}

\section{Controllability properties of linear systems\label{section3}}

Next we consider linear control systems in $\mathbb{K}^{d}$, $\mathbb{K}%
=\mathbb{R}$ or $\mathbb{K}=\mathbb{C}$, of the form
\begin{equation}
x_{k+1}=Ax_{k}+Bu_{k},\ \ u_{k}\in U\subset\mathbb{K}^{m},\label{linsys}%
\end{equation}
where $A\in Gl(d,\mathbb{K})$ and $B\in\mathbb{K}^{d\times m}$ and the control
range $U$ is a compact convex neighborhood of $0\in\mathbb{K}^{m}$ with
$U=\overline{\mathrm{int}U}$.

For initial value $x\in\mathbb{K}^{d}$ and control $u\in\mathcal{U}%
=U^{\mathbb{N}_{0}}$ the solutions of (\ref{linsys}) are given by
\[
\varphi(k,x,u)=A^{k}x+\sum_{i=0}^{k-1}A^{k-1-i}Bu_{i},k\in\mathbb{N}_{0}.
\]
Where convenient, we also use the notation $\varphi_{k,u}:=\varphi
(k,\cdot,u):\mathbb{R}^{d}\rightarrow\mathbb{R}^{d}$. Note the following observation.

\begin{proposition}
For $x\in\mathbb{K}^{d}$ the reachable set $\mathbf{R}_{k}(x)$ at time $k$,
\[
\mathbf{R}_{k}(x)=\{y\in\mathbb{K}^{d}\left\vert \ \exists u\in\mathcal{U}%
\ \mbox{with}\ \varphi(k,x,u)=y\right.  \}
\]
is compact and convex.
\end{proposition}

\begin{proof}
Convexity follows from the convexity of $U$. Since $U\subset\mathbb{K}^{m}$ is
compact, there is $M>0$ such that $\Vert u\Vert\leq M$, for all $u\in U$.
Then, if $y=\varphi(k,x,u)\in\mathbf{R}_{k}(x)$, $u=(u_{i})\in U^{k}$, we get
\[
\Vert y\Vert\leq\Vert A^{k}x\Vert+\sum_{i=0}^{k-1}\Vert A^{k-1-i}Bu_{i}%
\Vert\leq\Vert A\Vert^{k}\Vert x\Vert+M\sum_{i=0}^{k-1}\Vert A\Vert
^{k-1-i}\Vert B\Vert<\infty,
\]
hence $\mathbf{R}_{k}(x)$ is bounded. In order to show that $\mathbf{R}%
_{k}(x)$ is closed, consider a sequence $y_{n}=\varphi(k,x,u^{n})$ in
$\mathbf{R}_{k}(x)$ such that $y_{n}\rightarrow y\in\mathbb{K}^{d}$ and
$u^{n}\in U^{k}$. By compactness of $U$, we have that $U^{k}$ is compact,
hence there is a subsequence converging to some $u\in U^{k}$. Therefore
$y=\varphi(k,x,u)\in\mathbf{R}_{k}(x)$ by continuity.
\end{proof}

\begin{proposition}
\label{proposition_two}For all $k,l\in\mathbb{N}$ we have
\[
\mathbf{R}_{k}(0)+A^{k}\mathbf{R}_{l}(0)=\mathbf{R}_{l+k}(0)\text{ and
}\mathrm{int}\mathbf{R}_{k}\mathbf{(0)+}A^{k}\mathbf{R}_{l}(0)\subset
\mathrm{int}\mathbf{R}_{k+l}(0).
\]

\end{proposition}

\begin{proof}
Let $x_{1}\in\mathbf{R}_{k}(0)$ and $x_{2}\in\mathbf{R}_{l}(0)$. Then there
are $u,v\in\mathcal{U}$ such that
\[
x_{1}=\sum_{i=0}^{k-1}A^{k-1-i}Bu_{i}\ \mbox{and}\ x_{2}=\sum_{i=0}%
^{l-1}A^{l-1-i}Bv_{i}.
\]
Define
\[
w_{i}=\left\{
\begin{array}
[c]{rcl}%
v_{i}, & \mbox{if} & 0\leq i\leq l-1\\
u_{i-l}, & \mbox{if} & l\leq i\leq k+l-1
\end{array}
\right.  .
\]
Then
\begin{align*}
\varphi(k+l,0,w)  &  =\sum_{i=0}^{k+l-1}A^{k+l-1-i}Bw_{i}=\sum_{i=0}%
^{l-1}A^{k+l-1-i}Bw_{i}+\sum_{i=l}^{k+l-1}A^{k+l-1-i}Bw_{i}\\
&  =A^{k}\sum_{i=0}^{l-1}A^{l-1-i}Bv_{i}+\sum_{i=0}^{k-1}A^{k-1-i}Bu_{i}%
=A^{k}x_{2}+x_{1}.
\end{align*}
Hence $x_{1}+A^{k}x_{2}=\varphi(k+l,0,w)\in\mathbf{R}_{l+k}(0)$. The converse
inclusion follows by reversing these steps. The second assertion follows since
the set on left hand side is open.
\end{proof}

Define the time reversed counterpart of system (\ref{linsys}) by
\begin{equation}
x_{k+1}=A^{-1}x_{k}-A^{-1}Bu_{k},\ \ u_{k}\in U\subset\mathbb{K}%
^{m}.\label{linsys_rev}%
\end{equation}
The reachable and controllable sets from the origin at time $k$ for this
system are denoted by $\mathbf{R}_{k}^{-}(0)$ and $\mathbf{C}_{k}^{-}(0)$, respectively.

\begin{proposition}
\label{prop1.1.8}The reachable and controllable sets for system (\ref{linsys})
and the time reversed system (\ref{linsys_rev}) satisfy for all $k\in
\mathbb{N}$%
\[
\mathbf{R}_{k}(0)=\mathbf{C}_{k}^{-}(0)\text{ and }\mathbf{C}_{k}%
(0)=\mathbf{R}_{k}^{-}(0)\text{.}%
\]

\end{proposition}

\begin{proof}
Note that $x\in\mathbf{C}_{k}(0)$ if and only if there is $u\in\mathcal{U}$
with
\[
\ A^{k}x+\sum_{i=0}^{k-1}A^{k-1-i}Bu_{i}=0\text{, i.e., }\ x=-\sum_{i=0}%
^{k-1}A^{-1-i}Bu_{i}.
\]
For any $u\in U^{k}$, we define $v_{j}=u_{k-1-j}$, $0\leq j\leq k-1$. Then
\begin{align*}
x  &  =-\sum_{i=0}^{k-1}A^{-1-i}Bu_{i}=-\sum_{j=0}^{k-1}A^{-1-(k-1-j)}%
Bu_{k-1-j}=-\sum_{j=0}^{k-1}(A^{-1})^{k-j}Bv_{j}\\
&  =-\sum_{j=0}^{k-1}(A^{-1})^{k-1-j}A^{-1}Bv_{j}=\sum_{j=0}^{k-1}%
(A^{-1})^{k-1-j}(-A^{-1}B)v_{j}.
\end{align*}
Hence we conclude that $x\in\mathbf{C}_{k}(0)$ if and only if there exists a
control $v\in U^{k}$ such that $x=\varphi^{-}(k,0,v)$, where $\varphi^{-}$ is
the solution of (\ref{linsys_rev}). This proves that $\mathbf{C}%
_{k}(0)=\mathbf{R}_{k}^{-}(0)$. The other equality follows analogously.
\end{proof}

\begin{lemma}
\label{lemma_long}If $(A,B)$ is controllable, there is $\delta>0$ such that
the ball $B_{\delta}(0)$ satisfies $B_{\delta}(0)\subset\mathrm{int}%
\mathbf{R}_{d-1}(0)$. Furthermore, $\mathbf{R}_{n}(0)\subset\mathbf{R}_{m}(0)$
for $m\geq n$.
\end{lemma}

\begin{proof}
Since the control range is a neighborhood of $0$, controllability implies that
there is $\delta>0$ with $B_{\delta}(0)\subset\mathrm{int}\mathbf{R}_{d-1}%
(0)$. The second assertion follows since $0$ is an equilibrium for $u=0$.
\end{proof}

\begin{proposition}
\label{prop1.1.11}If $(A,B)$ is controllable, the reachable set of system
(\ref{linsys}) satisfies $\overline{\mathbf{R}(0)}=\overline{\mathrm{int}%
\mathbf{R}(0)}$.
\end{proposition}

\begin{proof}
The inclusion $\overline{\mathrm{int}\mathbf{R}(0)}\subset\overline
{\mathbf{R}(0)}$ holds trivially. For the converse we first show that
$\mathbf{R}(y)\subset\mathrm{int}\mathbf{R}(0)$ for $y\in\mathrm{int}%
\mathbf{R}(0)$. In fact, let there exists a neighborhood $V_{y}$ of $y\ $such
that $V_{y}\subset\mathbf{R}(0)$. Given $z\in\mathbf{R}(y)$, there are
$k\in\mathbb{N}$ and $u\in\mathcal{U}$ such that $z=\varphi(k,y,u)$. Since
$A\in Gl(d,\mathbb{R})$, the map $\varphi_{k,u}$ is a diffeomorphism and we
have that $\varphi_{k,u}(V_{y})$ is a neighborhood of $z$ and clearly
$\varphi_{k,u}(\mathbf{R}(0))\subset\mathbf{R}(0)$. So $z\in\varphi
_{k,u}(V_{y})\subset\mathbf{R}(0)$, which shows that $z\in\mathrm{int}%
\mathbf{R}(0)$.

Now, let $x\in\overline{\mathbf{R}(0)}$ and $V$ a neighborhood of $x$. There
is $y\in\mathbf{R}(0)$ such that $y\in V$, so there are $k\in\mathbb{N}$ and
$u\in\mathcal{U}$ such that $y=\varphi(k,0,u)$. Since $0\in\mathrm{int}%
\mathbf{R}(0)$ there exists a neighborhood $W$ of $0$ such that $W\subset
\mathrm{int}\mathbf{R}(0)$ and $\varphi_{k,u}(W)\subset V$ by continuity of
$\varphi_{k,u}$. For $z\in W$ the arguments above show that $\mathbf{R}%
(z)\subset\mathrm{int}\mathbf{R}(0)$ and it follows that
\[
\varphi(k,z,u)\in V\cap\mathbf{R}(z)\subset V\cap\mathrm{int}\mathbf{R}(0)
\]
and hence $x\in\overline{\mathrm{int}\mathbf{R}(0)}$.
\end{proof}

We will need the following lemmas.

\begin{lemma}
\label{lemma2}For every $\lambda\in\mathbb{C}$ there are $n_{k}\rightarrow
\infty$ such that $\frac{\lambda^{n_{k}}}{\left\vert \lambda\right\vert
^{n_{k}}}\rightarrow1$, and, in particular,
\[
\frac{\operatorname{Im}(\lambda^{n_{k}})}{\operatorname{Re}(\lambda^{n_{k}}%
)}\rightarrow0\text{ for }k\rightarrow\infty.
\]

\end{lemma}

\begin{proof}
There is $\theta\in\lbrack0,2\pi)$ with $\lambda=\left\vert \lambda\right\vert
(\cos\theta+\imath\sin\theta)$, hence%
\[
\lambda^{n}=\left\vert \lambda\right\vert ^{n}(\cos(n\theta)+\imath
\sin(n\theta)).
\]
If $\theta\in2\pi\mathbb{Q}$, there are $n,N\in\mathbb{N}$ with $n\theta
=N2\pi$, hence $\lambda^{n}=\left\vert \lambda\right\vert ^{n}\cos
(N2\pi)=\left\vert \lambda\right\vert ^{n}$. Else, there are $n_{k}%
\rightarrow\infty$ such that modulo $2\pi$ one has $n_{k}\theta\rightarrow0$.
This implies $\cos(n_{k}\theta)\rightarrow1$ and $\sin(n_{k}\theta
)\rightarrow0$, hence%
\[
\frac{\lambda^{n_{k}}}{\left\vert \lambda\right\vert ^{n_{k}}}=\cos
(n_{k}\theta)+\imath\sin(n_{k}\theta)\rightarrow1.
\]
This implies%
\[
\frac{\operatorname{Im}(\lambda^{n_{k}})}{\operatorname{Re}(\lambda^{n_{k}}%
)}=\frac{\operatorname{Im}\left(  \frac{\lambda^{n_{k}}}{\left\vert
\lambda\right\vert ^{n_{k}}}\right)  }{\operatorname{Re}\left(  \frac
{\lambda^{n_{k}}}{\left\vert \lambda\right\vert ^{n_{k}}}\right)  }=\frac
{\sin(n_{k}\theta)}{\cos(n_{k}\theta)}\rightarrow0.
\]

\end{proof}

The next lemma states a property of convex sets.

\begin{lemma}
\label{Lemma_cone}If $C$ is an open convex subset of $\mathbb{K}^{n}$ and
$Y\subset C$ a subspace, then $C=C+Y$.
\end{lemma}

The following theorem describes the general structure of reachable and
controllable sets. It is analogous to a well known property of linear systems
in continuous time, cf. Sontag \cite[Section 3.6]{Son98} and Hinrichsen and
Pritchard \cite[Theorem 6.2.15]{HiP20}; the proof for discrete-time systems,
however, is more involved. Recall that the state space $\mathbb{K}^{d}$ can be
decomposed with respect to $A$ into the direct sum of the stable subspace
$E^{s}$, the center space $E^{c}$ and the unstable subspace $E^{u}$ which are
the direct sums of all generalized (real) eigenspaces for the eigenvalues
$\lambda$ of $A$ with $\left\vert \lambda\right\vert <1$, $\left\vert
\lambda\right\vert =1$ and $\left\vert \lambda\right\vert >1$, respectively.
Furthermore, we let $E^{uc}:=E^{u}\oplus E^{c}$ and $E^{sc}:=E^{s}\oplus
E^{c}$.

\begin{theorem}
\label{prop1.1.12} Consider the control system given by (\ref{linsys}) and
suppose that the system without control restriction is controllable.

(i) There exists a compact and convex set $K\subset E^{s}\subset\mathbb{K}%
^{d}$ with nonvoid interior with respect to $E^{s}$ such that $\overline
{\mathbf{R}(0)}=K+E^{uc}$. Moreover $0\in K$ and $E^{uc}\subset\mathrm{int}%
\mathbf{R}(0)$.

(ii) There exists a compact and convex set $F\subset E^{u}\subset
\mathbb{K}^{d}$ with nonvoid interior with respect to $E^{u}$ such that
$\overline{\mathbf{C}(0)}=F+E^{sc}$. Moreover $0\in F$ and $E^{sc}%
\subset\mathrm{int}\mathbf{C}(0)$.
\end{theorem}

\begin{proof}
We will first prove the result for $\mathbb{K}=\mathbb{C}$.

(i) In the first step, we will show that $E^{uc}\subset\mathrm{int}%
\mathbf{R}(0)$. As $\mathbf{R}(0)$ is convex, its interior is convex too.
Therefore it suffices to prove that the generalized eigenspaces for
eigenvalues with absolute value greater than or equal to $1$ are contained in
$\mathrm{int}\mathbf{R}(0)$. Fix an eigenvalue $\lambda$ of $A$ with
$\left\vert \lambda\right\vert \geq1$ and let $E_{q}(\lambda)=\mathrm{ker}%
(A-\lambda I)^{q}$, $q\in\mathbb{N}_{0}$. It suffices to show that
$E_{q}(\lambda)\subset\mathrm{int}\mathbf{R}(0)$ for all $q$.

We prove the statement by induction on $q$, the case $q=0$ being trivial since
$E_{q}(\lambda)=\{0\}\subset\mathrm{int}\mathbf{R}(0)$. So assume that
$E_{q-1}(\lambda))\subset\mathrm{int}\mathbf{R}(0)$ and take any $w\in
E_{q}(\lambda)$. We must show that $w\in\mathrm{int}\mathbf{R}(0)$. By Lemma
\ref{lemma_long} there is $\delta>0$ such that $aw\in\mathrm{int}%
\mathbf{R}_{d-1}(0)$ for all $a\in\mathbb{C}$ with $\left\vert a\right\vert
<\delta$.

Note that for all $\left\vert a\right\vert <\delta$ and all $n\geq1$%
\begin{align*}
A^{n}aw  &  =(A-\lambda I+\lambda I)^{n}aw=\sum_{j=0}^{n}{\binom{n}{j}%
}(A-\lambda I)^{n-j}\lambda^{j}aw\\
&  =\lambda^{n}aw+\sum_{j=0}^{n-1}{\binom{n}{j}}(A-\lambda I)^{n-j}\lambda
^{j}aw.
\end{align*}
Since $aw\in E_{q}(\lambda)$, it follows that $(A-\lambda I)^{i}aw\in
E_{q-1}(\lambda)$ for all $i\geq1$, hence $z(n):=\sum_{j=0}^{n-1}{\binom{n}%
{j}}(A-\lambda I)^{n-j}\lambda^{j}aw\in E_{q-1}(\lambda),n\geq1$. Using
$aw\in\mathrm{int}\mathbf{R}_{d-1}(0)$ Lemma \ref{lemma_long} and Lemma
\ref{Lemma_cone} imply for$\ n\geq1$%
\begin{equation}
\lambda^{n}aw=A^{n}aw-z(n)\in A^{n}aw+E_{q-1}(\lambda)\subset\mathrm{int}%
\mathbf{R}_{n+d-1}(0)+E_{q-1}(\lambda)\subset\mathrm{int}\mathbf{R}(0).
\label{3.3}%
\end{equation}
We write%
\[
a=\alpha+\imath\beta\text{ and }\lambda^{n}=x_{n}+\imath y_{n}%
\]
with $\alpha,\beta\in\mathbb{R}$ and $x_{n},y_{n}\in\mathbb{R}^{d}$.

\textbf{Claim:} There are a sequence $(n_{k})_{k\in\mathbb{N}}$ with
$n_{k}\rightarrow\infty$ and $a_{n_{k}}\in\mathbb{C}$ with $\left\vert
a_{n_{k}}\right\vert <\delta$ such that $\lambda^{n_{k}}a_{n_{k}}\in
\mathbb{R}$.

In\ fact, we have%
\[
\lambda^{n}a=(x_{n}+\imath y_{n})(\alpha+\imath\beta)=x_{n}\alpha-y_{n}%
\beta+\imath(x_{n}\beta+y_{n}\alpha)\in\mathbb{R},
\]
if and only if $x_{n}\beta+y_{n}\alpha=0$.

Case (a): If $x_{n}=0$, one may choose $\alpha_{n}:=0$ and gets $\lambda
^{n}a_{n}=-y_{n}\beta_{n}\in\mathbb{R}$ for $\beta_{n}=\frac{\delta}{2}$ with
$\left\vert a_{n}\right\vert =\left\vert \beta_{n}\right\vert =\frac{\delta
}{2}$.

Case (b): Otherwise $\lambda^{n}a\in\mathbb{R}$ if and only if%
\[
\beta=-\alpha\frac{y_{n}}{x_{n}}=-\alpha\frac{\operatorname{Im}(\lambda^{n}%
)}{\operatorname{Re}(\lambda^{n})}.
\]
According to Lemma \ref{lemma2} there are $n_{k}\in\mathbb{N}$, arbitrarily
large, such that with $\alpha_{n_{k}}:=\frac{\delta}{2}$ and $\beta_{n_{k}%
}:=-\alpha_{n_{k}}\frac{y_{n_{k}}}{x_{n_{k}}}$
\[
\left\vert \beta_{n_{k}}\right\vert =\frac{\delta}{2}\left\vert \frac
{\operatorname{Im}(\lambda^{n_{k}})}{\operatorname{Re}(\lambda^{n_{k}}%
)}\right\vert <\frac{\delta}{2}.
\]
It follows for $a_{n_{k}}:=\alpha_{n_{k}}+\beta_{n_{k}}$ that%
\[
\left\vert a_{n_{k}}\right\vert ^{2}=\alpha_{n_{k}}^{2}+\beta_{n_{k}}%
^{2}<\frac{1}{4}\delta^{2}+\frac{1}{4}\delta^{2}\text{, and hence }\left\vert
a_{n_{k}}\right\vert <\delta.
\]
We have shown that with this choice of $a_{n_{k}}$ we have $\lambda^{n_{k}%
}a_{n_{k}}\in\mathbb{R}$ and the \textbf{Claim} is proved. Furthermore in case
(a), by $\left\vert \lambda\right\vert \geq1$,%
\[
\left\vert \lambda^{n}a_{n}\right\vert =\left\vert \lambda\right\vert
^{n}\left\vert a_{n}\right\vert \geq\left\vert a_{n}\right\vert =\frac{\delta
}{2},
\]
and in case (b)%
\[
\left\vert \lambda^{n_{k}}a_{n_{k}}\right\vert =\left\vert \lambda\right\vert
^{n_{k}}\left\vert a_{n_{k}}\right\vert \geq\left\vert a_{n_{k}}\right\vert
\geq\left\vert \alpha_{n_{k}}\right\vert =\frac{\delta}{2}.
\]
Now choose $\ell\in\mathbb{N}$ with $\ell\geq2/\delta$. Recall that all points
$a_{n_{k}}w\in\mathrm{int}\mathbf{R}_{d-1}(0)$. We may assume that $n_{2}\geq
n_{1}+d-1$, hence%
\[
A^{n_{1}}a_{n_{1}}w\in\mathrm{int}\mathbf{R}_{n_{1}+d-1}(0)\subset
\mathrm{int}\mathbf{R}_{n_{2}}(0).
\]
We may also assume that $n_{3}-n_{2}\geq n_{2}+d-1$, hence%
\[
A^{n_{2}}a_{n_{2}}w\in\mathrm{int}\mathbf{R}_{n_{2}+d-1}(0)\subset
\mathrm{int}\mathbf{R}_{n_{3}-n_{2}}(0).
\]
Thus Proposition \ref{proposition_two} implies%
\[
A^{n_{1}}a_{n_{1}}w+A^{n_{2}}a_{n_{2}}w\in\mathrm{int}\mathbf{R}_{n_{2}%
}(0)+A^{n_{2}}\mathbf{R}_{n_{3}-n_{2}}(0)\subset\mathrm{int}\mathbf{R}%
_{n_{3}-n_{2}+n_{2}}(0)=\mathrm{int}\mathbf{R}_{n_{3}}(0).
\]
Proceeding in this way, we finally arrive at%
\[
\sum_{k=1}^{\ell}A^{n_{k}}a_{n_{k}}w\in\mathrm{int}\mathbf{R}_{n_{\ell}}(0).
\]
Thus we find with (\ref{3.3}),%
\[
\sum_{k=1}^{\ell}\lambda^{n_{k}}a_{n_{k}}w=\sum_{k=1}^{\ell}\left[  A^{n_{k}%
}a_{n_{k}}w-z(n_{k})\right]  \in\mathrm{int}\mathbf{R}_{n_{\ell}}%
(0)+E_{q-1}(\lambda)\subset\mathrm{int}\mathbf{R}(0).
\]
If $\lambda^{n_{k}}a_{n_{k}}>0$ for all $k\in\{1,\dotsc,\ell\}$, then (the
real number)
\[
\sum_{k=1}^{\ell}\lambda^{n_{k}}a_{n_{k}}>\ell\cdot\delta/2\geq1.
\]
For the $k$ with $\lambda^{n_{k}}a_{n_{k}}<0$, replace $a_{n_{k}}$ by
$-a_{n_{k}}$, to get the same conclusion. This shows that $w$ is a convex
combination of the points $0$ and $\sum_{k=1}^{\ell}\lambda^{n_{k}}a_{n_{k}}w$
in $\mathrm{int}\mathbf{R}(0)$, thus convexity of this set implies
$w\in\mathrm{int}\mathbf{R}(0)$ completing the induction step $E_{q}%
(\lambda)\subset\mathrm{int}\mathbf{R}(0)$. Hence we have shown that
$E^{uc}\subset\mathrm{int}\mathbf{R}(0)$.\medskip

It remains to construct a set $K$ as in the assertion. Define $K_{0}%
:=\mathrm{int}\mathbf{R}(0)\cap E^{s}$. Then it follows that
\[
K_{0}+E^{uc}=(\text{$\mathrm{int}$}\mathbf{R}(0)\cap E^{s})+E^{uc}%
\subset\text{$\mathrm{int}$}\mathbf{R}(0)+E^{uc}\subset\text{$\mathrm{int}$%
}\mathbf{R}(0).
\]
For the converse inclusion, let $v\in\mathrm{int}\mathbf{R}(0)$, then $v=x+y$
where $x\in E^{s}$ and $y\in E^{uc}$, hence by Lemma \ref{Lemma_cone},
\[
x=v-y\in\text{$\mathrm{int}$}\mathbf{R}(0)+E^{uc}=\text{$\mathrm{int}$%
}\mathbf{R}(0),
\]
which shows that $x\in K_{0}$ and therefore $v\in K_{0}+E^{s}$. This shows
that%
\begin{equation}
K_{0}+E^{uc}=\text{$\mathrm{int}$}\mathbf{R}(0). \label{eq1}%
\end{equation}
In order to show that $K_{0}$ is bounded, consider the projection
$\pi:\mathbb{C}^{d}=E^{s}\oplus E^{uc}\rightarrow$ $E^{s}$ along $E^{uc}$.
Since $E^{s}$ and $E^{uc}$ are $A$-invariant, $\pi$ commutes with $A$ and we
have $\pi A^{n}=A^{n}\pi$, for all $n\in\mathbb{N}_{0}$. For each $x\in
K_{0}=\mathrm{int}\mathbf{R}(0)\cap E^{s}$, there are $k\in\mathbb{N}$ and
$u=(u_{i})\in\mathcal{U}$ such that
\[
x=\sum_{i=0}^{k-1}A^{k-1-i}Bu_{i}.
\]
Since $A|_{E^{s}}$ is a linear contraction, there exist constants $a\in(0,1)$
and $c\geq1$ such that $\Vert A^{n}x\Vert\leq ca^{n}\Vert x\Vert$ for all
$n\in\mathbb{N}$ and $x\in E^{s}$. Since $U$ is compact, there is $M>0$ such
that $\Vert\pi Bu\Vert\leq M$, for all $u\in U$, so
\[
x=\pi(x)=\pi\left(  \sum_{i=0}^{k-1}A^{k-1-i}Bu_{i}\right)  =\sum_{i=0}%
^{k-1}\pi A^{k-1-i}Bu_{i}=\sum_{i=0}^{k-1}A^{k-1-i}\pi Bu_{i},
\]
hence
\[
\Vert x\Vert\leq\sum_{i=0}^{k-1}\left\Vert A^{k-1-i}\pi Bu_{i}\right\Vert
\leq\sum_{i=0}^{k-1}\left\Vert A^{k-1-i}\Vert\Vert\pi Bu_{i}\right\Vert \leq
cM\sum_{i=0}^{k-1}a^{k-1-i}=cM\dfrac{1-a^{k}}{1-a}%
\]
showing that $K_{0}$ is bounded. As a consequence, $K:=\overline{K_{0}%
}=\overline{\mathrm{int}\mathbf{R}(0)\cap E^{s}}$ is a compact convex set
which has nonvoid interior relative to $E^{s}$. Moreover, $K+E^{uc}$ is
closed, because $K$ is compact. Therefore it follows from Proposition
\ref{prop1.1.11} and (\ref{eq1}) that
\[
\overline{\mathbf{R}(0)}=\overline{\text{$\mathrm{int}$}\mathbf{R}%
(0)}=\overline{K_{0}+E^{uc}}=K+E^{uc}.
\]

(ii) Consider the time reversed system (\ref{linsys_rev}). Note that
$\mathbb{C}^{d}=E_{-}^{s}\oplus E_{-}^{c}\oplus E_{-}^{u}$, where $E_{-}^{s}$,
$E_{-}^{c}$ and $E_{-}^{u}$ are the sums of the generalized eigenspaces for
the eigenvalues $\mu$ of $A^{-1}$ with $\left\vert \mu\right\vert <1$,
$\left\vert \mu\right\vert =1$ and $\left\vert \mu\right\vert >1$,
respectively. Now $\lambda$ is an eigenvalue of $A$ (note that $\lambda\neq0$
since $A\in Gl(d,\mathbb{C})$), if and only if $\mu=\lambda^{-1}$ is an
eigenvalue of $A^{-1}$. Hence we have $E_{-}^{s}=E^{u}$, $E_{-}^{c}=E^{c}$ and
$E_{-}^{u}=E^{s}$. By (i) there exists a compact and convex set $F\subset
\mathbb{C}^{d}$ which has nonvoid interior with respect to $E_{-}^{s}=E^{u}$
such that $\overline{\mathbf{R}^{-}(0)}=F+E_{-}^{uc}$, $0\in F$ and
$E_{-}^{uc}\subset\mathrm{int}\mathbf{R}^{-}(0)$. By Proposition
\ref{prop1.1.8},
\[
E^{sc}=E_{-}^{uc}\subset\text{$\mathrm{int}$}\mathbf{R}^{-}%
(0)=\text{$\mathrm{int}$}\mathbf{C}(0)
\]
and
\[
\overline{\mathbf{C}(0)}=F+E_{-}^{uc}=F+E^{sc}.
\]
This completes the proof of the theorem for the case $\mathbb{K}=\mathbb{C}$.

It remains to prove the theorem for the case $\mathbb{K}=\mathbb{R}$. Note
that if $A\in Gl(d,\mathbb{R})$, then $u-\imath v\in E^{s},u,v\in
\mathbb{R}^{d}$, implies $u+\imath v,v+\imath u\in E^{s}$ and a similar
implication holds for $E^{uc}$. Hence%
\begin{align}
\operatorname{Re}E^{s} &  =E^{s}\cap\mathbb{R}^{d},\operatorname{Re}%
E^{uc}=E^{uc}\cap\mathbb{R}^{d},\label{HP16}\\
E^{s} &  =\operatorname{Re}E^{s}+\imath\operatorname{Re}E^{s},E^{uc}%
=\operatorname{Re}E^{uc}\oplus\imath\operatorname{Re}E^{uc}\nonumber
\end{align}
Let $U_{\mathbb{C}}:=U+\imath U$ and apply the result above for $\mathbb{K}%
=\mathbb{C}$. Clearly $(A,B)$ is controllable, when considered as a system
with state space $\mathbb{C}^{d}$ and $U_{\mathbb{C}}$ is a convex compact
neighborhood of $0\in\mathbb{C}^{m}$ with $U_{\mathbb{C}}\subset
\overline{\mathrm{int}U_{\mathbb{C}}}$.

Denote the reachable set from $0$ of the real and complex system by
$\mathbf{R}_{\mathbb{R}}$ and $\mathbf{R}_{\mathbb{C}}$, respectively. It
follows from the complex version of the theorem that the compact convex set
$K_{\mathbb{C}}:=\overline{\mathrm{int}(\mathbf{R}_{\mathbb{C}})\cap E^{s}}$
has non-empty interior relative to $E^{s}$ and satisfies $\overline
{\mathbf{R}_{\mathbb{C}}}=K_{\mathbb{C}}\cap E^{uc}$. Since every
$u\in\mathcal{U}_{\mathbb{C}}$ is of the form $u=v+\imath w$, where
$v,w\in\mathcal{U}$, and $\varphi(k,0,u)=\varphi(k,0,v)+\imath\varphi
(k,0,w),k\in\mathbb{N}$, we have%
\begin{equation}
\mathcal{U}_{\mathbb{C}}=\mathcal{U}_{\mathbb{R}}+\imath\mathcal{U}%
_{\mathbb{R}}\text{ and }\mathbf{R}_{\mathbb{C}}=\mathbf{R}_{\mathbb{R}%
}+\imath\mathbf{R}_{\mathbb{R}}. \label{HP20}%
\end{equation}
It follows that%
\[
\mathbf{R}_{\mathbb{R}}=\operatorname{Re}\mathbf{R}_{\mathbb{C}}%
,\mathrm{int}\mathbf{R}_{\mathbb{R}}=\operatorname{Re}\mathrm{int}%
\mathbf{R}_{\mathbb{C}},
\]
where the interior of $\mathbf{R}_{\mathbb{R}}$ is relative to $\mathbb{R}%
^{d}$ and the interior of $\mathbf{R}_{\mathbb{C}}$ is relative to
$\mathbb{C}^{d}$. Now, if $W,Z$ $\subset\mathbb{C}^{d}$ are subsets of the
form%
\[
W=W_{1}+\imath W_{2},Z=Z_{1}+\imath Z_{2},
\]
where $W_{1},W_{2},Z_{1},Z_{2}\subset\mathbb{R}^{d}$ and $W\cap Z\not =%
\varnothing$, then $W\cap Z=\left(  W_{1}\cap Z_{1}\right)  +\imath\left(
W_{2}\cap Z_{2}\right)  $ and so $\operatorname{Re}(W\cap Z)=\operatorname{Re}%
W\cap\operatorname{Re}Z$. Applying this equality to $W=\mathrm{int}%
\mathbf{R}_{\mathbb{C}}$ and $Z=E^{s}$ we obtain from (\ref{HP20}) and
(\ref{HP16}) that%
\[
K=\overline{(\operatorname{Re}(\mathrm{int}\mathbf{R}_{\mathbb{C}}%
))\cap\operatorname{Re}E^{s}}=\overline{\operatorname{Re}(\mathrm{int}%
\mathbf{R}_{\mathbb{C}})\cap E^{s})}=\operatorname{Re}K_{\mathbb{C}}.
\]
Hence $K$ is a compact convex subset of $\mathbb{R}^{d}$, which has a
non-empty interior relative to $\operatorname{Re}E^{s}$. Using (\ref{HP20})
for the second equality we get%
\[
\overline{\mathbf{R}_{\mathbb{R}}}=\overline{\operatorname{Re}\mathbf{R}%
_{\mathbb{C}}}=\operatorname{Re}\overline{\mathbf{R}_{\mathbb{C}}%
}=\operatorname{Re}(K_{\mathbb{C}}+E^{u,s})=K+\operatorname{Re}E^{u,s}.
\]
This concludes the proof.
\end{proof}

Next we present a necessary and sufficient condition for controllability in
$\mathbb{K}^{d}$. This consequence of Theorem \ref{prop1.1.12} illustrates
that controllability only holds under very strong assumptions on the spectrum
of the matrix $A$. In the next section, we will instead consider subsets of
the state space where complete controllability holds, i.e., control sets.
Recall that the system without control restriction is controllable in
$\mathbb{R}^{d}$ if and only if $(A,B)$ is controllable.

\begin{corollary}
\label{corollary16}Consider the discrete-time linear system given in
(\ref{linsys}).

(i) The reachable set $\mathbf{R}(0)=\mathbb{K}^{d}$ if and only if $(A,B)$ is
controllable and $A$ has no eigenvalues with absolute value less than $1$.

(ii) The controllable set $\mathbf{C}(0)=\mathbb{K}^{d}$ if and only if
$(A,B)$ is controllable and $A$ has no eigenvalues with absolute value greater
than $1$.

(iii) The system is controllable in $\mathbb{K}^{d}$ if and only if $(A,B)$ is
controllable and all eigenvalues of $A$ have absolute value equal to $1$.
\end{corollary}

\begin{proof}
(i) If $\mathbf{R}(0)=\mathbb{K}^{d}$, then the pair $(A,B)$ is controllable,
since $\mathbf{R}(0)$ is contained in the image of Kalman's matrix
$[B\ \ AB\ \ \ldots\ \ A^{d-1}B]$. Moreover, if there is an eigenvalue
$\lambda$ of $A$ with $|\lambda|<1$, then $E^{s}\neq\{0\}$ and $E^{u}$ is a
proper subset of $\mathbb{K}^{d}$. By Theorem \ref{prop1.1.12} (ii), there is
a nonvoid compact set $F\subset E^{u}$ such that $E^{sc}+F=\overline
{\mathbf{R}(0)}=\mathbb{K}^{d}$, a contradiction.

Conversely, if $(A,B)$ is controllable and all eigenvalues $\lambda$ of $A$
satisfy $\left\vert \lambda\right\vert \geq1$, then by Theorem
\ref{prop1.1.12} (i) we have $\mathbb{K}^{d}=E^{uc}\subset\mathrm{int}%
\mathbf{R}(0)\subset\mathbf{R}(0)$.

(ii) This follows analogously.

(iii) This is a consequence of assertions (i) and (ii) observing that
$\mathbf{R}(0)=\mathbf{C}(0)=\mathbb{K}^{d}$ holds if and only if for all
$x,y\in\mathbb{K}^{d}$ there are a control $u\in\mathcal{U}$ and a time
$k\in\mathbb{N}$ with $\varphi(k,x,u)=y$.
\end{proof}

\begin{remark}
In the continuous-time case, a result analogous to Corollary \ref{corollary16}
is given e.g. in Sontag \cite[Section 3.6]{Son98}. For the discrete-time case,
we are not aware of a result in the literature covering Corollary
\ref{corollary16}. In the special case of two inputs (i.e., $m=2$) the
characterization of null-controllability in Corollary \ref{corollary16} (ii)
is given in Wing and Desoer \cite[Section V, Theorem 2]{WinD63}.
\end{remark}

\section{Control sets for linear systems\label{section4}}

Next we analyze linear control systems in $\mathbb{R}^{d}$ of the form%
\begin{equation}
x_{k+1}=Ax_{k}+Bu_{k},u_{k}\in U\subset\mathbb{R}^{m} \label{lin}%
\end{equation}
with $A\in Gl(d,\mathbb{R})$ and $B\in\mathbb{R}^{d\times m}$ and suppose that
$U$ is a convex compact neighborhood of $0\in\mathbb{R}^{m}$ with
$U=\overline{\mathrm{int}U}$. Recall that the system without control
restrictions is controllable in $\mathbb{R}^{d}$ if and only if $\mathrm{rank}%
[B~AB\dotsc A^{d-1}B]=d$, i.e., the pair $(A,B)$ is controllable.

\begin{theorem}
\label{theorem_existence}There exists a unique control set $D$ with nonvoid
interior of system (\ref{lin}) if and only if the system without control
restriction is controllable in $\mathbb{R}^{d}$. In this case $0\in D_{0}%
\cap\mathrm{int}D$.
\end{theorem}

\begin{proof}
The controllability condition for $(A,B)$ is necessary for the existence of
$D$, since it guarantees that accessibility condition (\ref{access0}) holds
for all $x\in\mathbb{R}^{d}$ and, for the system without control constraints,
the reachable and the null-controllable subspaces coincide with $\mathbb{R}%
^{d}$. Since $0\in\mathrm{int}U$, one verifies that for $k\geq d-1$
\[
0\in\mathrm{int}(\mathbf{C}_{k}(0))\cap\mathrm{int}(\mathbf{R}_{k}%
(0))=:D^{\prime}.
\]
Then every point $x\in D^{\prime}$ can be steered to any other point $z\in
D^{\prime}$ (first steer $x$ to the origin in time $k$ and then the origin to
$z$ in time $k$) and $0\in\mathrm{int}(\mathbf{C}(0))$. As in the proof of
Proposition \ref{lem1.1.2} one finds that $D^{\prime}$ is contained in a
control set $D$. Thus we have established the existence of a control set $D$
with nonvoid interior, and $0\in D_{0}\cap\mathrm{int}D$. It remains to show uniqueness.

Let $\tilde{D}\subset\mathbb{R}^{d}$ be an arbitrary control set with nonvoid
interior. By Proposition \ref{lem1.1.2} its transitivity set $\tilde{D}_{0}$
is nonvoid and for $x_{0}\in\tilde{D}_{0}$%
\[
\tilde{D}=\overline{\mathbf{R}(x_{0})}\cap\mathbf{C}(x_{0}).
\]
By linearity, we have $\varphi(k,x_{1},u)=x_{2}$ for $k\in\mathbb{N}$ and
$x_{1},x_{2}\in\mathbb{R}^{d}$ implies $\varphi(k,\alpha x_{1},\alpha
u)=\alpha x_{2}$ for any $\alpha\in(0,1]$. Here the control $\alpha u$ has
values in $U$, since $U$ is convex and $0\in U$. This implies that
$\alpha\tilde{D}$ is contained in some control set $D^{\alpha}$ and
$\mathrm{int}(\alpha\tilde{D})$ is contained in the interior of $D^{\alpha}$.
Now choose any $x\in\mathrm{int}\tilde{D}$ and suppose, by way of
contradiction, that
\[
\alpha_{0}:=\inf\{\alpha\in(0,1]\left\vert \forall\beta\in\lbrack
\alpha,1]:\beta x\in\tilde{D}\right.  \}>0.
\]
Then $\alpha_{0}x\in\partial\tilde{D}$ and $\alpha_{0}x\in\mathrm{int}%
D^{\alpha_{0}}$. Therefore $\tilde{D}\cap\mathrm{int}D^{\alpha_{0}}%
\not =\varnothing$, and it follows that $\tilde{D}=D^{\alpha_{0}}$ and
$\alpha_{0}x\in\mathrm{int}\tilde{D}$. This is a contradiction and so
$\alpha_{0}=0$. Choosing $\alpha>0$ small enough such that $\alpha x\in D$, we
obtain $\alpha x\in\tilde{D}\cap D\not =\varnothing$. Now it follows that
$\tilde{D}=D$.
\end{proof}

The following theorem gives a spectral characterization of boundedness of the
control set. Recall that $A$ is called hyperbolic if all eigenvalues $\lambda$
of $A$ satisfy $\left\vert \lambda\right\vert \neq1$.

\begin{theorem}
\label{theorem_bounded}Assume that $(A,B)$ is controllable. Then the control
set $D$ with nonvoid interior of system (\ref{lin}) is bounded if and only if
$A$ is hyperbolic.
\end{theorem}

\begin{proof}
By Theorem \ref{prop1.1.12} there are compact sets $K\subset E^{s}$, $F\subset
E^{u}$ such that
\[
\overline{\mathbf{R}(0)}=K+E^{c}+E^{u}\ \mbox{ and }\ \overline{\mathbf{C}%
(0)}=F+E^{c}+E^{s}.
\]
By Proposition \ref{lem1.1.2}, $D=\overline{\mathbf{R}(0)}\cap\mathbf{C}(0)$,
because $0\in D_{0}\subset\mathrm{int}D$, and hence every element $x\in D$ can
be represented in the following two ways:
\[
x=k+x_{1}+x_{+}=f+x_{1}+x_{-},
\]
where $k\in K\subset E^{s}$, $f\in F\subset E^{u}$, $x_{1}\in E^{c}$,
$x_{-}\in E^{s}$ and $x_{+}\in E^{u}$. Since $\mathbb{R}^{d}=E^{s}\oplus
E^{c}\oplus E^{u}$ we get $k=x_{-}$, $f=x_{+}$. As $E^{c}=E^{sc}\cap
E^{uc}\subset\mathbf{R}(0)\cap\mathbf{C}(0)\subset D$, we conclude that
$E^{c}\subset D\subset K+E^{c}+F$, and so the control set $D$ is bounded if
and only if $E^{c}=\{0\}$.
\end{proof}

\begin{remark}
We know that in the hyperbolic case%
\begin{equation}
D=K_{0}+F^{\prime}\label{D_sum}%
\end{equation}
with $K_{0}\subset E^{s},F^{\prime}\subset F\subset E^{u}$, where $K_{0}$ and
$F$ are compact sets with $0\in K_{0}\cap F$. In particular, it follows that
$K_{0},F^{\prime}\subset D$.
\end{remark}

Next we present a simple example illustrating control sets.

\begin{example}
Consider for $d=2$ and $m=1$%
\[
\left[
\begin{array}
[c]{c}%
x_{k+1}\\
y_{k+1}%
\end{array}
\right]  =\left[
\begin{array}
[c]{cc}%
2 & 0\\
0 & \frac{1}{2}%
\end{array}
\right]  \left[
\begin{array}
[c]{c}%
x_{k}\\
y_{k}%
\end{array}
\right]  +\left[
\begin{array}
[c]{c}%
1\\
1
\end{array}
\right]  u_{k},~u_{k}\in U=[-1,1].
\]
We claim that for this hyperbolic matrix $A$ the unique control set with
nonvoid interior is $D=(-1,1)\times\lbrack-2,2]$. The stable subspace
associated with the eigenvalue $\frac{1}{2}$ of $A$ is the $y$-axis, the
unstable subspace associated with the eigenvalue $2$ is the $x$-axis. For a
constant control $u\in\lbrack-1,1]$, one computes the equilibrium as
$(x(u),y(u))^{\top}=(u,2u)^{\top}$. In particular. for $u=1$ and $u=-1$ one
obtains the equilibria%
\[
\left[
\begin{array}
[c]{c}%
x(1)\\
y(1)
\end{array}
\right]  =\left[
\begin{array}
[c]{c}%
-1\\
2
\end{array}
\right]  \text{ and }\left[
\begin{array}
[c]{c}%
x(-1)\\
y(-1)
\end{array}
\right]  =\left[
\begin{array}
[c]{c}%
1\\
-2
\end{array}
\right]  ,
\]
resp. It is clear that for all $u\in(-1,1)$ the equilibrium $(-u,2u)^{\top}$
is in the interior of the control set $D$. Furthermore, observe that for
$x_{0}>1$ one has in the next step $2x_{0}+u>x_{0}$ and for $x_{0}<-1$ one has
$2x_{0}+u<x_{0}$. If $y_{0}>2$, then $\frac{1}{2}y_{0}+u<\frac{1}{2}%
y_{0}+1\leq y_{0}$ and if $y_{0}<-2$, then $\frac{1}{2}y_{0}+u\geq\frac{1}%
{2}y_{0}-1>y_{0}$. Hence solutions starting left of the vertical line $x=-1$
and right of $x=1$ have to go to the left and to the right, respectively.
Solutions which start above the horizontal line$\ y=2$ and below $y=-2$, have
to go down and up, respectively. This shows that the control set must be
contained in $(-1,1)\times\lbrack-2,2]$. The controllability property within
$D$ can be seen by the following analysis. If we start in an equilibrium
$(x(\alpha),y(\alpha))^{\top}=(-\alpha,2\alpha)^{\top},\alpha\in\left(
-1,1\right)  $, we get e.g.%
\[
\left[
\begin{array}
[c]{c}%
x_{1}\\
y_{1}%
\end{array}
\right]  =\left[
\begin{array}
[c]{c}%
-2\alpha\\
\alpha
\end{array}
\right]  +\left[
\begin{array}
[c]{c}%
1\\
1
\end{array}
\right]  u_{0},~\left[
\begin{array}
[c]{c}%
x_{2}\\
y_{2}%
\end{array}
\right]  =\left[
\begin{array}
[c]{c}%
-4\alpha\\
\frac{1}{2}\alpha
\end{array}
\right]  +\left[
\begin{array}
[c]{c}%
2\\
\frac{1}{2}%
\end{array}
\right]  u_{0}+\left[
\begin{array}
[c]{c}%
1\\
1
\end{array}
\right]  u_{1}.
\]
For the reachable set, we see that after one step the line segment
$S=\{(u,u)^{\top},\allowbreak u\in\lbrack-1,1]\}$ is shifted to $(-2\alpha
,\alpha)^{\top}$. After two time steps the line segment $S$ is shifted to
$(-4\alpha,\frac{1}{2}a)^{\top}$ and at every point the line segment
$\{(2u,\frac{1}{2}u)^{\top}\left\vert u\in\lbrack-1,1]\right.  \}$ is added.
One can show that the equilibrium $(0,0)^{\top}$ can be reached. If we start
in $(0,0)^{\top}$, we compute
\begin{align*}
\left[
\begin{array}
[c]{c}%
x_{1}\\
y_{1}%
\end{array}
\right]   &  =\left[
\begin{array}
[c]{c}%
1\\
1
\end{array}
\right]  u_{0},\left[
\begin{array}
[c]{c}%
x_{2}\\
y_{2}%
\end{array}
\right]  =\left[
\begin{array}
[c]{c}%
2\\
\frac{1}{2}%
\end{array}
\right]  u_{0}+\left[
\begin{array}
[c]{c}%
1\\
1
\end{array}
\right]  u_{1},\\
\left[
\begin{array}
[c]{c}%
x_{3}\\
y_{3}%
\end{array}
\right]   &  =\left[
\begin{array}
[c]{c}%
4\\
\frac{1}{4}%
\end{array}
\right]  u_{0}+\left[
\begin{array}
[c]{c}%
2\\
\frac{1}{2}%
\end{array}
\right]  u_{1}+\left[
\begin{array}
[c]{c}%
1\\
1
\end{array}
\right]  u_{2}.
\end{align*}
Proceeding in this way one finds that one can get approximately to all points
in $D$ and, in particular, to the equilibria $(-1,2)^{\top}$ and
$(1,-2)^{\top}$. Connecting appropriately the controls, one finally shows that
$D=(-1,1)\times\lbrack-2,2]$ is a control set.
\end{example}

\section{Invariance pressure\label{section5}}

In this section we recall the concept of invariance pressure considered in
\cite{Cocosa1}, \cite{Cocosa2}, \cite{ZHuag19} where potentials are defined on
the control range. Furthermore, we introduce the generalized version of total
invariance pressure, where the potentials are defined on the product of the
state space and the control range. Again we consider the general system
(\ref{nonlinear}).

A pair $(K,Q)$ of nonvoid subsets of $M$ is called admissible if $K\subset Q$
is compact and for each $x\in K$ there exists $u\in\mathcal{U}$ such that
$\varphi(\mathbb{N},x,u)\subset Q$. For an admissible pair $(K,Q)$ and
$\tau>0$, a $(\tau,K,Q)$-spanning set $\mathcal{S}$ of controls is a subset of
$\mathcal{U}$ such that for all $x\in K$ there is $u\in\mathcal{S}$ with
$\varphi(k,x,u)\in Q$ for all $k\in\left\{  1,\dotsc,\tau\right\}  $. Denote
by $C(U,\mathbb{R})$ the set of continuous function $f:U\rightarrow\mathbb{R}$
which we call potentials.

For a potential $f\in C(U,\mathbb{R})$ denote $(S_{\tau}f)(u):=\sum
_{i=0}^{\tau-1}f(u_{i}),u\in\mathcal{U}$, and
\[
a_{\tau}(f,K,Q)=\inf\left\{  \sum_{u\in\mathcal{S}}e^{(S_{\tau}f)(u)}%
\left\vert \mathcal{S}\text{ }(\tau,K,Q)\text{-spanning}\right.  \right\}  .
\]

\begin{definition}
The invariance pressure $P_{inv}(f,K,Q)$ of control system (\ref{nonlinear})
is defined by%
\[
P_{inv}(f,K,Q)=\overline{\underset{\tau\rightarrow\infty}{\lim}}\frac{1}{\tau
}\log a_{\tau}(f,K,Q).
\]

\end{definition}

For the potential $f=\mathbf{0}$, this reduces to the notion of invariance
entropy, $P_{inv}(\mathbf{0},K,Q)=h_{inv}(K,Q)$.

In order to define the total invariance pressure associate to every control
$u$ in a $(\tau,K,Q)$-spanning set $\mathcal{S}$ of controls an initial value
$x_{u}\in K$ with $\varphi(k,x_{u},u)\in Q$ for all $k\in\left\{
1,\dotsc,\tau\right\}  $. Then a set of state-control pairs of the form%
\[
\mathcal{S}_{tot}=\{(x_{u},u)\in K\times\mathcal{S}\left\vert \varphi
(k,x_{u},u)\in Q\text{ for all }k\in\left\{  1,\dotsc,\tau\right\}  \right.
\}
\]
is called totally $(\tau,K,Q)$-spanning. Denote by $C(Q\times U,\mathbb{R})$
the set of continuous function $f:Q\times U\rightarrow\mathbb{R}$ which we
again call potentials. For a potential $f\in C(Q\times U,\mathbb{R})$ and
$(x,u)\in M\times\mathcal{U}$ denote $(S_{\tau}f)(x,u):=\sum_{i=0}^{\tau
-1}f(\varphi(i,x,u),u_{i})$ and
\[
a_{\tau}(f,K,Q):=\inf\left\{  \sum_{(x,u)\in\mathcal{S}_{tot}}e^{(S_{\tau
}f)(x,u)}\left\vert \mathcal{S}_{tot}\text{ totally }(\tau
,K,Q)\text{-spanning}\right.  \right\}  .
\]

\begin{definition}
The total invariance pressure $P_{tot}(f,K,Q;\Sigma)$ of control system
(\ref{nonlinear}) is defined by%
\begin{equation}
P_{tot}(f,K,Q)=\underset{\tau\rightarrow\infty}{\overline{\lim}}\frac{1}{\tau
}\log a_{\tau}(f,K,Q). \label{tip}%
\end{equation}

\end{definition}

Note that by continuity and monotonicity of the logarithm,%
\begin{align}
& P_{tot}(f,K,Q)\label{tip_alt}\\
& =\underset{\tau\rightarrow\infty}{\overline{\lim}}\inf\left\{  \frac{1}%
{\tau}\log\sum_{(x,u)\in\mathcal{S}_{tot}}e^{(S_{\tau}f)(x,u)}\left\vert
\mathcal{S}_{tot}\text{ totally }(\tau,K,Q)\text{-spanning}\right.  \right\}
.\nonumber
\end{align}
Furthermore $-\infty<a_{\tau}(f,K,Q)\leq\infty$ for every $\tau\in\mathbb{N}$,
every admissible pair $(K,Q)$, and every potential $f$ if every countable
totally spanning set contains a finite totally spanning subset, cf.
\cite[Remark 7]{Cocosa2}. If $f(x,u)$ is independent of $x$, i.e., it is a
continuous function on $U$, the total invariance pressure coincides with the
invariance pressure.

\begin{remark}
The definition of totally $(\tau,K,Q)$-spanning sets is inspired by the
definition of spanning sets for $(K,Q)$ in Wang, Huang, and Sun \cite[p.
313]{WangHS19}, where a similar notion is introduced in the context of
invariant partitions which provide an alternative definition of invariance entropy..
\end{remark}

The next elementary proposition presents some properties of the function
$P_{tot}(\cdot,K,Q):C(Q\times U,\mathbb{R})\rightarrow\mathbb{R}\cup
\{\pm\infty\}$.

\begin{proposition}
\label{propert}The following assertions hold for an admissible pair $(K,Q)$,
functions $f,g\in C(Q\times U,\mathbb{R})$ and $c\in\mathbb{R}$:

(i) For $f\leq g$ one has $P_{tot}(f,K,Q)\leq P_{tot}(g,K,Q)$.

(ii) $P_{tot}(f+c,K,Q)=P_{tot}(f,K,Q)+c$.
\end{proposition}

\begin{proof}
This follows easily from the definition, cf. also \cite[Proposition
13]{Cocosa1}.
\end{proof}

The following proposition shows that, in the definition of total invariance
pressure, we can take the limit superior over times which are integer
multiples of some fixed time step $\tau\in\mathbb{N}$. The proof is analogous
to the proof given in \cite[Theorem 20]{Cocosa2} for invariance pressure of
continuous-time systems.

\begin{proposition}
\label{discretization}For all $f\in C(Q\times U,\mathbb{R})$ with
$\inf_{(x,u)\in Q\times U}f(x,u)>-\infty$ the total invariance pressure
satisfies for $\tau\in\mathbb{N}$%
\[
P_{tot}(f,K,Q)=\underset{n\rightarrow\infty}{\overline{\lim}}\frac{1}{n\tau
}\log a_{n\tau}(f,K,Q).
\]

\end{proposition}

\begin{proof}
For every $f\in C(Q\times U,\mathbb{R})$, the inequality%
\begin{equation}
P_{tot}(f,K,Q)\geq\underset{n\rightarrow\infty}{\overline{\lim}}\frac{1}%
{n\tau}\log a_{n\tau}(f,K,Q) \label{4.1c}%
\end{equation}
is obvious. For the converse note that the function $g(x,u):=f(x,u)-\inf f$ is
nonnegative (if $f\geq0$, we may consider $f$ instead of $g$). Let $\tau
_{k}\in(0,\infty)$ with $\tau_{k}\rightarrow\infty$ for $k\rightarrow\infty$.
Then for every $k\geq1$ there exists $n_{k}\in\mathbb{N}_{0}$ such that
$n_{k}\tau\leq\tau_{k}<(n_{k}+1)\tau$ and $n_{k}\rightarrow\infty$ for
$k\rightarrow\infty$. Since $g\geq0$ it follows that%
\[
a_{\tau_{k}}(g,K,Q)\leq a_{(n_{k}+1)\tau}(g,K,Q)
\]
and consequently
\[
\frac{1}{\tau_{k}}\log a_{\tau_{k}}(g,K,Q)\leq\frac{1}{n_{k}\tau}\log
a_{(n_{k}+1)\tau}(g,K,Q).
\]
This yields
\[
\underset{k\rightarrow\infty}{\overline{\lim}}\frac{1}{\tau_{k}}\log
a_{\tau_{k}}(g,K,Q)\leq\underset{k\rightarrow\infty}{\overline{\lim}}\frac
{1}{n_{k}\tau}\log a_{(n_{k}+1)\tau}(g,K,Q).
\]
Since $\frac{1}{n_{k}\tau}=\frac{n_{k}+1}{n_{k}}\frac{1}{(n_{k}+1)\tau}$ and
$\frac{n_{k}+1}{n_{k}}\rightarrow1$ for $k\rightarrow\infty$, we obtain
\begin{align*}
\underset{k\rightarrow\infty}{\overline{\lim}}\frac{1}{\tau_{k}}\log
a_{\tau_{k}}(g,K,Q)  &  \leq\underset{k\rightarrow\infty}{\overline{\lim}%
}\frac{1}{(n_{k}+1)\tau}\log a_{(n_{k}+1)\tau}(g,K,Q)\\
&  \leq\underset{n\rightarrow\infty}{\overline{\lim}}\frac{1}{n\tau}\log
a_{n\tau}(g,K,Q).
\end{align*}
Together with Proposition \ref{propert} (ii) and (\ref{4.1c}) applied to
$f-\inf f$, this shows that%
\begin{align*}
P_{tot}(f,K,Q)  &  =P_{tot}(f-\inf f,K,Q)+\inf f\\
&  =\underset{n\rightarrow\infty}{\overline{\lim}}\frac{1}{n\tau}\log
a_{n\tau}(f-\inf f,K,Q)+\inf f\\
&  =\underset{n\rightarrow\infty}{\overline{\lim}}\frac{1}{n\tau}\log
a_{n\tau}(f,K,Q).
\end{align*}

\end{proof}

The following result is given in \cite[Corollary 15]{Cocosa2} for
continuous-time systems. The discrete-time case is proved analogously.

\begin{proposition}
\label{compact}Let $K_{1},K_{2}$ be two compact sets with nonvoid interior
contained in a control set $D\subset M$ and assume that every point in $D$ is
accessible. Then $(K_{1},D)$ and $(K_{2},D)$ are admissible pairs and for all
$f\in C(U,\mathbb{R})$ we have
\[
P_{inv}(f,K_{1},D)=P_{inv}(f,K_{2},D).
\]

\end{proposition}

\section{Invariance pressure for linear systems\label{section6}}

The main result of this section presents a formula for the invariance pressure
of the unique control set with nonvoid interior for hyperbolic linear control
systems of the form (\ref{lin}).

We start with a proposition providing an upper bound for the total invariance
pressure of the unique control set with nonvoid interior, cf. Theorems
\ref{theorem_existence} and \ref{theorem_bounded}. The proof uses arguments
from \cite{Cocosa3} which in turn are based on a construction by Kawan
\cite[Theorem 4.3]{Kawa11b}, \cite[Theorem 5.1]{Kawa13} (for the discrete-time
case cf. also \cite[Remark 5.4]{Kawa13} and Nair, Evans, Mareels, Moran
\cite[Theorem 3]{NEMM04}).

Let $A^{+}$ be the restriction of $A$ to the unstable subspace $E^{u}$. The
unstable determinant of $A$ is%
\[
\det A^{+}=\prod\limits_{\lambda\in\sigma(A)}\lambda^{n_{\lambda}}\text{ and
}\log\left\vert \det A^{+}\right\vert =\sum_{\lambda\in\sigma(A)}n_{\lambda
}\max\{0,\log\left\vert \lambda\right\vert \},
\]
where $n_{\lambda}$ denotes the algebraic multiplicity of an eigenvalue
$\lambda$ of $A$.

\begin{proposition}
\label{prop_upper_tot}Consider a linear control system of the form (\ref{lin})
and assume that the pair $(A,B)$ is controllable with a hyperbolic matrix $A$.
Let $D$ be the unique control set with nonvoid interior and let $f\in
C(\overline{D}\times U,\mathbb{R})$. Then there exists a compact set $K\subset
D$ with nonvoid interior such that the total invariance pressure satisfies%
\[
P_{tot}(f,K,D)\leq\log\left\vert \det A^{+}\right\vert +\inf_{(\tau,x,u)}%
\frac{1}{\tau}\sum_{i=0}^{\tau-1}f(\varphi(i,x,u),u_{i}),
\]
where the infimum is taken over all $\tau\in\mathbb{N}$ with $\tau\geq d$ and
all $\tau$-periodic controls $u$ with a $\tau$-periodic trajectory
$\varphi(\cdot,x,u)$ in $\mathrm{int}D$ such that $u_{i}\in\mathrm{int}U$ for
$i\in\{0,\dotsc,\tau-1\}$.
\end{proposition}

\begin{proof}
We will construct a compact subset $K\subset D$ with nonvoid interior such
that the inequality above holds. Observe that then by Proposition
\ref{compact} the pair $(K,D)$ is admissible.

We may suppose that $A$ has real Jordan form $R=T^{-1}AT$. In fact, writing
$x=Tx^{\prime}$ one obtains%
\begin{equation}
x_{k+1}^{\prime}=T^{-1}ATx_{k}^{\prime}+T^{-1}Bu_{k}=Rx_{k}^{\prime}%
+B^{\prime}u_{k}\label{transform}%
\end{equation}
with $B^{\prime}:=T^{-1}B$. Then with $f^{\prime}(x^{\prime},u)=f(Tx^{\prime
},u)=:f(x,u),K^{\prime}:=T^{-1}K$, and $D^{\prime}:=T^{-1}D$ the total
invariance pressure $P_{toz}(f,K,D)$ coincides with the total invariance
pressure $P_{tot}(f^{\prime},K^{\prime},D^{\prime})$ of (\ref{transform}).
Consider a $\tau^{0}$-periodic control $u^{0}(\cdot)$ with $\tau^{0}$-periodic
trajectory $\varphi(\cdot,x^{0},u^{0})$ as in the statement of the theorem,
hence%
\begin{equation}
x^{0}=R^{\tau^{0}}x^{0}+\sum_{i=0}^{\tau^{0}-1}R^{\tau^{0}-i}B^{\prime}%
u_{i}.\label{periodic}%
\end{equation}

\textbf{Step 1:} Choose a basis $\mathcal{B}$ of $\mathbb{R}^{d}$ adapted to
the real Jordan structure of $R$ and let $L_{1}(R),\dotsc,L_{r}(R)$ be the
Lyapunov spaces of $R$, that is, the sums of the generalized eigenspaces
corresponding to eigenvalues $\lambda$ with the absolute value $\left\vert
\lambda\right\vert =\rho_{j}$. This yields the decomposition%
\[
\mathbb{R}^{d}=L_{1}(R)\oplus\cdots\oplus L_{r}(R).
\]
Let $d_{j}=\dim L_{j}(R)$ and denote the restriction of $R$ to $L_{j}(R)$ by
$R_{j}$. Now take an inner product on $\mathbb{R}^{d}$ such that the basis
$\mathcal{B}$ is orthonormal with respect to this inner product and let
$\left\Vert \cdot\right\Vert $ denote the induced norm.

\textbf{Step 2:} We fix some constants: Let $S_{0}$ be a real number which
satisfies%
\[
S_{0}>\sum\limits_{j=1}^{r}\max\{1,d_{j}\rho_{j}\}=\log\left\vert \det
A^{+}\right\vert ,
\]
and choose $\xi=\xi(S_{0})>0$ such that%
\[
0<d\xi<S_{0}-\sum\limits_{j=1}^{r}\max\{1,d_{j}\rho_{j}\}
\]
and such that $\rho_{j}<1$ implies $\rho_{j}+\xi<1$ for all $j$. Let
$\delta\in(0,\xi)$. It follows that there exists a constant $c=c(\delta)\geq1$
such that for all $j\ $and for all $k\in\mathbb{N}$%
\[
\left\Vert R_{j}^{k}\right\Vert \leq c(\rho_{j}+\delta)^{k}.
\]
For every $m\in\mathbb{N}$ we define positive integers by%
\[
M_{j}(m):=\left\{
\begin{array}
[c]{ccc}%
\left\lfloor (\rho_{j}+\xi)^{m}\right\rfloor +1 & \text{if} & \rho_{j}\geq1\\
1 & \text{if} & \rho_{j}<1
\end{array}
\right.
\]
and a function $\beta:\mathbb{N}\rightarrow(0,\infty)$ by%
\[
\beta(m):=\max_{1\leq j\leq r}\left\{  (\rho_{j}+\delta)^{m}\frac{\sqrt{d_{j}%
}}{M_{j}(m)}\right\}  ,m\in\mathbb{N}.
\]
If $\rho_{j}<1$, then $\rho_{j}+\delta<1$ and $M_{j}(m)\equiv1$, and hence
$(\rho_{j}+\delta)^{m}/M_{j}(m)$ converges to zero for $m\rightarrow\infty$.
If $\rho_{j}\geq1$, we have $M_{j}(m)\geq(\rho_{j}+\xi)^{m}$ and hence
\begin{equation}
(\rho_{j}+\delta)^{m}\frac{\sqrt{d_{j}}}{M_{j}(m)}\leq(\rho_{j}+\delta
)^{m}\frac{\sqrt{d_{j}}}{(\rho_{j}+\xi)^{m}}=\left(  \frac{\rho_{j}+\delta
}{\rho_{j}+\xi}\right)  ^{m}\sqrt{d_{j}}. \label{beta2}%
\end{equation}
Since $\delta\in(0,\xi)$, we have $\frac{\rho_{j}+\delta}{\rho_{j}+\xi}<1$
showing that also in this case $\beta(m)\rightarrow0$ for $m\rightarrow\infty$.

Since we assume controllability of $(A,B)$ and $\tau^{0}\geq d$ there exists
$C_{0}>0$ such that for every $x\in\mathbb{R}^{d}$ there is a control
$u\in\mathcal{U}$ with%
\begin{equation}
\varphi(\tau^{0},x,u)=R^{\tau^{0}}x+\sum_{i=0}^{\tau^{0}-1}R^{\tau^{0}%
-i}B^{\prime}u_{i}=0\text{ and }\left\Vert u\right\Vert _{\infty}\leq
C_{0}\left\Vert x\right\Vert . \label{Kawan5.9MODIFIED0}%
\end{equation}
The inequality follows by the inverse mapping theorem. For the corresponding
trajectory we find a constant $C_{1}>0$ such that for $k\in\{1,\ldots,\tau
^{0}\}$%
\begin{equation}
\left\Vert \varphi(k,x,u)\right\Vert \leq\left\Vert R\right\Vert
^{k}\left\Vert x\right\Vert +\sum_{i=0}^{k-1}\left\Vert R\right\Vert
^{k-i}\left\Vert B^{\prime}\right\Vert C_{0}\left\Vert x\right\Vert \leq
C_{1}\left\Vert x\right\Vert . \label{Kawan5.9MODIFIED}%
\end{equation}
For $b_{0}>0$ let $\mathcal{C}$ be the $d$-dimensional compact cube
$\mathcal{C}$ in $\mathbb{R}^{d}$ centered at the origin with sides of length
$2b_{0}$ parallel to the vectors of the basis $B$. Choose $b_{0}$ small enough
such that
\[
K:=x^{0}+\mathcal{C}\subset D
\]
and $\overline{B(u^{0}(k),Cb_{0})}\subset U$ for all $k\in\{0,\dotsc,\tau
^{0}\}$. This is possible, since $x^{0}\in\mathrm{int}D$ and all values
$u^{0}(k)$ are in the interior of $U$.

\textbf{Step 3.} Let $\varepsilon>0$ and $\tau=m\tau^{0}$ with $m\in
\mathbb{N}$. \ By Theorem \ref{theorem_bounded}, the closure $\overline{D}$ is
compact, hence for the continuous function $f$ on the compact set
$\overline{D}\times U$ there is $\varepsilon_{1}>0$ such that for all
$(x,u),(x^{\prime},u^{\prime})\in\overline{D}\times U$%
\begin{equation}
\max\left\{  \left\Vert x-x^{\prime}\right\Vert ,\left\Vert u-u^{\prime
}\right\Vert \right\}  <\varepsilon_{1}\text{ implies }\left\vert
f(x,u)-f(x^{\prime},u^{\prime})\right\vert <\varepsilon. \label{one}%
\end{equation}
We may take $m\in\mathbb{N}$ large enough such that%
\begin{equation}
\frac{d}{\tau}\log2=\frac{d}{m\tau^{0}}\log2<\varepsilon. \label{two}%
\end{equation}
Furthermore, we may choose $b_{0}$ small enough such that%
\begin{equation}
C_{0}b_{0}<\varepsilon_{1}\text{ and }C_{1}b_{0}<\varepsilon_{1}.
\label{three}%
\end{equation}
Partition $\mathcal{C}$ by dividing each coordinate axis corresponding to a
component of the $j$th Lyapunov space $L_{j}(R)$ into $M_{j}(\tau)$ intervals
of equal length. The total number of subcuboids in this partition of
$\mathcal{C}$ is $\prod_{j=1}^{r}M_{j}(\tau)^{d_{j}}$. Next we will show that
it suffices to take $\prod_{j=1}^{r}M_{j}(\tau)^{d_{j}}$ control functions to
steer the system from all states in $x^{0}+\mathcal{C}$ back to $x^{0}%
+\mathcal{C}$ in time $\tau$ such that the controls are within distance
$\varepsilon_{1}$ to $u^{0}$ and the corresponding trajectories remain within
distance $\varepsilon_{1}$ from the trajectory $\varphi(\cdot,x^{0},u^{0})$.
Let $y$ be the center of a subcuboid. By (\ref{Kawan5.9MODIFIED0}) there
exists $u=(u_{0},\ldots,u_{\tau^{0}-1})$ such that%
\begin{equation}
\varphi(\tau^{0},y,u)=0\text{ and }\left\Vert u\right\Vert _{\infty}\leq
C_{0}\left\Vert y\right\Vert \leq C_{0}b_{0}<\varepsilon_{1}. \label{four}%
\end{equation}
For $k\geq t_{0}$ let $u_{k}=0$. Hence $\varphi(\tau,y,u)=0$ and $u(t)\in U$
for all $k\in\{0,\dotsc,\tau\}$. Using (\ref{periodic}) and linearity, we find
that $x^{0}+y$ is steered by $u^{0}+u$ in time $\tau=m\tau^{0}$ to $x^{0}$,%
\begin{equation}
\varphi(\tau,x^{0}+y,u^{0}+u)=\varphi(\tau,x^{0},u^{0})+\varphi(\tau
,y,u)=x^{0}. \label{periodic2}%
\end{equation}
Now consider an arbitrary point $x\in\mathcal{C}$. Then it lies in one of the
subcuboids and we denote the corresponding center of this subcuboid by $y$
with associated control $u=u(y)$. We will show in \textbf{Step 4} that
$u^{0}+u$ also steers $x^{0}+x$ back to $x^{0}+\mathcal{C}$ and in
\textbf{Step 5} that the corresponding trajectory $\varphi(k,x^{0}+x,u^{0}+u)$
remains within distance $\varepsilon_{1}$ of $\varphi(k,x^{0},u^{0}%
),k\in\{0,\ldots,\tau\}$.

\textbf{Step 4.} Observe that%
\[
\left\Vert x-y\right\Vert \leq\frac{b_{0}}{M_{j}(\tau)}\sqrt{d_{j}}.
\]
By (\ref{beta2}) this implies that%
\[
\left\Vert R^{\tau}x-R^{\tau}y\right\Vert \leq\left\Vert R_{j}^{m\tau^{0}%
}\right\Vert \left\Vert x-y\right\Vert \leq c(\rho_{j}+\delta)^{m\tau^{0}%
}\frac{b_{0}}{M_{j}(m\tau^{0})}\sqrt{d_{j}}\rightarrow0\text{ for
}m\rightarrow\infty,
\]
and hence for $m$ large enough $\left\Vert R^{\tau}x-R^{\tau}y\right\Vert \leq
b_{0}$. This implies that the solution $\varphi(k,x^{0}+x,u^{0}+u),k\in
\mathbb{N}$, satisfies for $m$ large enough by (\ref{periodic2}) and
linearity,%
\begin{align*}
&  \left\Vert \varphi(\tau,x^{0}+x,u^{0}+u)-x^{0}\right\Vert \\
&  =\left\Vert R^{\tau}(x^{0}+x)+\sum_{i=0}^{\tau-1}R^{\tau-i}B^{\prime}%
(u_{i}^{0}+u_{i})-x^{0}\right\Vert \\
&  \leq\left\Vert R^{\tau}(x^{0}+x)-R^{\tau}(x^{0}+y)\right\Vert +\left\Vert
R^{\tau}(x^{0}+y)+\sum_{i=0}^{\tau-1}R^{\tau-i}B^{\prime}(u_{i}^{0}%
+u_{i})-x^{0}\right\Vert \\
&  \leq\left\Vert R^{\tau}x-R^{\tau}y\right\Vert +\left\Vert \varphi
(\tau,x^{0}+y,u^{0}+u)-x^{0}\right\Vert \\
&  \leq b_{0}+0.
\end{align*}
This shows that $\varphi(\tau,x^{0}+x,u^{0}+u)\in x^{0}+\mathcal{C}$ and it
also follows that $\varphi(\tau,x^{0}+x,u^{0}+u)\in D$ for all $k\in
\{0,1,\ldots,\tau\}$.

\textbf{Step 5. }By linearity and formulas (\ref{Kawan5.9MODIFIED0}),
(\ref{Kawan5.9MODIFIED}), and (\ref{three}) we can estimate for $k\in
\{0,1,\ldots,\tau^{0}\}$%
\begin{align*}
&  \left\Vert \varphi(k,x^{0}+x,u^{0}+u)-\varphi(k,x^{0},u^{0})\right\Vert \\
&  =\left\Vert R^{k}(x^{0}+x)+\varphi(k,0,u^{0}+u)-R^{k}x^{0}-\varphi
(k,0,u^{0})\right\Vert \\
&  =\left\Vert R^{k}x+\varphi(k,0,u)\right\Vert =\left\Vert \varphi
(k,x,u)\right\Vert \leq C_{1}\left\Vert x\right\Vert \leq C_{1}b_{0}%
<\varepsilon_{1}.
\end{align*}
Together with (\ref{four}) and (\ref{one}) this shows that for $k\in
\{0,1,\ldots,\tau\}$%
\begin{equation}
\left\vert f\left(  \varphi(k,x^{0}+x,u^{0}+u),u_{k}^{0}+u_{k})-f(\varphi
(k,x^{0},u^{0}),u_{k}^{0})\right)  \right\vert <\varepsilon. \label{five}%
\end{equation}

\textbf{Step 6. }We have constructed $\prod_{j=1}^{r}M_{j}(\tau)^{d_{j}}$
control functions that allow us to steer the system from all states in
$K=x^{0}+\mathcal{C}$ back to $x^{0}+\mathcal{C}$ in time $\tau$ and satisfy
(\ref{five}). By iterated concatenation of these control functions we obtain a
totally $(n\tau,K,D)$-spanning set $\mathcal{S}_{tot}$ for each $n\in
\mathbb{N}$ with cardinality%
\[
\#\mathcal{S}_{tot}\mathcal{=}\left(  \prod_{j=1}^{r}M_{j}(\tau)^{d_{j}%
}\right)  ^{n}=\left(  \prod_{j:\rho_{j}\geq0}\left(  \left\lfloor (\rho
_{j}+\xi)^{\tau}\right\rfloor +1\right)  ^{d_{j}}\right)  ^{n}.
\]
By (\ref{five}) it follows that%
\begin{align*}
\log a_{n\tau}(f,K,D) &  \leq\log\left(  \sum\nolimits_{(x,u)\in
\mathcal{S}_{tot}}e^{(S_{n\tau}f)(x,u)}\right)  \\
&  =\log\left(  \sum\nolimits_{(x,u)\in\mathcal{S}_{tot}}e^{(S_{n\tau}%
f)(x^{0},u^{0})}\cdot e^{(S_{n\tau}f)(x,u)-(S_{n\tau}f)(x^{0},u^{0})}\right)
\\
&  \leq\log\sum\nolimits_{(x,u)\in\mathcal{S}_{tot}}e^{(S_{n\tau}%
f)(x^{0},u^{0})}+\log e^{\sum_{i=0}^{n\tau-1}\varepsilon}\\
&  \leq\log\left(  \#\mathcal{S}_{tot}\mathcal{\cdot}e^{(S_{n\tau}%
f)(x^{0},u^{0})}\right)  +n\tau\varepsilon.
\end{align*}
This implies, using also (\ref{two}),%
\begin{align*}
\frac{1}{n\tau}\log a_{n\tau}(f,K,D) &  \leq\frac{1}{\tau}\sum_{j:\rho_{j}%
\geq0}d_{j}\log(\left\lfloor e^{(\rho_{j}+\xi)\tau}\right\rfloor +1)+\frac
{1}{n\tau}\sum_{i=0}^{n\tau-1}f(\varphi(i,x^{0},u^{0}),u_{i}^{0}%
)+\varepsilon\\
&  \leq\frac{1}{\tau}\sum_{j:\rho_{j}\geq0}d_{j}\log(2e^{(\rho_{j}+\xi)\tau
})+\frac{1}{\tau^{0}}\sum_{i=0}^{\tau^{0}-1}f(\varphi(i,x^{0},u^{0}),u_{i}%
^{0})+\varepsilon\\
&  \leq\frac{d}{\tau}\log2+\frac{1}{\tau}\sum_{j:\rho_{j}\geq0}d_{j}(\rho
_{j}+\xi)\tau+\frac{1}{\tau^{0}}\sum_{i=0}^{\tau^{0}-1}f(\varphi(i,x^{0}%
,u^{0}),u_{i}^{0})+\varepsilon\\
&  \leq\varepsilon+d\xi+\sum_{j:\rho_{j}\geq0}d_{j}\rho_{j}+\frac{1}{\tau^{0}%
}\sum_{i=0}^{\tau^{0}-1}f(\varphi(i,x^{0},u^{0}),u_{i}^{0})+\varepsilon\\
&  <S_{0}+\frac{1}{\tau^{0}}\sum_{i=0}^{\tau^{0}-1}f(\varphi(i,x^{0}%
,u^{0}),u_{i}^{0})+2\varepsilon.
\end{align*}
Since $\varepsilon$ can be chosen arbitrarily small and $S_{0}$ arbitrarily
close to $\log\left\vert \det A^{+}\right\vert $, the assertion of the
proposition follows.
\end{proof}

For the invariance pressure, we obtain the following consequence.

\begin{corollary}
\label{cor_upper}Consider a linear control system of the form (\ref{lin}) and
assume that the pair $(A,B)$ is controllable with a hyperbolic matrix $A$. Let
$D$ be the unique control set with nonvoid interior and let $f\in
C(U,\mathbb{R})$. Then for every compact set $K\subset D$ with nonvoid
interior the invariance pressure satisfies%
\[
P_{inv}(f,K,D)\leq\log\left\vert \det A^{+}\right\vert +\inf_{(\tau,x,u)}%
\frac{1}{\tau}\sum_{i=0}^{\tau-1}f(u_{i}),
\]
where the infimum is taken over all $\tau\in\mathbb{N}$ with $\tau\geq d$ and
all $\tau$-periodic controls $u$ with a $\tau$-periodic trajectory
$\varphi(\cdot,x,u)$ in $\mathrm{int}D$ such that $u_{i}\in\mathrm{int}U$ for
$i\in\{0,\dotsc,\tau-1\}$.
\end{corollary}

\begin{proof}
The assertion follows from Proposition \ref{prop_upper_tot}, since every
compact subset of $D$ is contained in a compact subset $K$ of $D$ with nonvoid
interior and the invariance pressure is independent of the choice of such a
set $K$ by Proposition \ref{compact}$.$
\end{proof}

\begin{remark}
Kawan \cite[Theorem 3.1]{Kawa13} derives for the outer invariance entropy
$h_{inv,out}(K,Q)$, which is a lower bound for the invariance entropy, the
formula%
\[
h_{inv,out}(K,Q)=\log\left\vert \det A^{+}\right\vert .
\]
Here $(K,Q)$ is an admissible pair, $K$ has positive Lebesgue measure, and $Q$
is compact. For the potential $f=0$, Corollary \ref{cor_upper} shows that the
invariance entropy satisfies%
\[
h_{inv}(K,Q)\leq\log\left\vert \det A^{+}\right\vert =h_{inv,out}(K,Q)\leq
h_{inv}(K,Q)
\]
implying that
\begin{equation}
h_{inv}(K,Q)=\log\left\vert \det A^{+}\right\vert .\label{h_inv}%
\end{equation}

\end{remark}

We proceed to prove a lower bound for the invariance pressure. Recall that
with respect to $A$ the state space $\mathbb{R}^{d}$ can be decomposed into
the direct sum of the center-stable subspace $E^{sc}$ and the unstable
subspace $E^{u}$ which are the direct sums of all generalized real eigenspaces
for the eigenvalues $\lambda$ with $\left\vert \lambda\right\vert \leq1$ and
$\left\vert \lambda\right\vert >1$, resp. Let $\pi:\mathbb{R}^{d}\rightarrow
E^{u}$ be the projection along $E^{sc}$.

\begin{proposition}
\label{prop_lower}Let $K\subset D$ be compact and assume that both $K$ and $D$
have positive and finite Lebesgue measure. Then for every $f\in C(U,\mathbb{R}%
)$
\[
P_{inv}(f,K,D)\geq\log\left\vert \det A^{+}\right\vert +\inf_{(\tau,x,u)}%
\frac{1}{\tau}\sum_{i=0}^{\tau-1}f(u_{i}),
\]
where the infimum is taken over all $(\tau,x,u)\in\mathbb{N}\times
D\times\mathcal{U}$ with $\tau\geq d$ and $\pi\varphi(i,x,u)\in\pi D$ for
$i\in\{0,1,\dotsc,\tau-1\}$.
\end{proposition}

\begin{proof}
Every $(\tau,K,Q)$-spanning set $\mathcal{S}$ satisfies%
\begin{equation}
\log\sum_{u\in\mathcal{S}}e^{(S_{\tau}f)(u)}\geq\log\inf_{u\in\mathcal{S}%
}e^{(S_{\tau}f)(u)}+\log\#\mathcal{S}. \label{24b}%
\end{equation}
First suppose that the unstable subspace of $A$ is trivial, $E^{u}=0$. Formula
(\ref{h_inv}) implies that
\[
\underset{\tau\rightarrow\infty}{\overline{\lim}}\frac{1}{\tau}\inf\left\{
\log\#\mathcal{S}\left\vert \mathcal{S}\text{ }(\tau,K,Q)\text{-spanning}%
\right.  \right\}  =h_{inv}(K,D)=\log\left\vert \det A^{+}\right\vert =0.
\]
Now (\ref{tip_alt}) and (\ref{24b}) implies%
\begin{align*}
&  P_{inv}(f,K,Q)=\underset{\tau\rightarrow\infty}{\overline{\lim}}\frac
{1}{\tau}\inf\left\{  \log\sum_{u\in\mathcal{S}}e^{(S_{\tau}f)(u)}\left\vert
\mathcal{S}\text{ }(\tau,K,Q)\text{-spanning}\right.  \right\} \\
&  \geq\underset{\tau\rightarrow\infty}{\overline{\lim}}\frac{1}{\tau}%
\inf\left\{  \log\inf_{u\in\mathcal{S}}e^{(S_{\tau}f)(u)}+\log\#\mathcal{S}%
\left\vert \mathcal{S}\text{ }(\tau,K,Q)\text{-spanning}\right.  \right\} \\
&  \geq\underset{\tau\rightarrow\infty}{\overline{\lim}}\frac{1}{\tau}%
\inf\left\{  \inf_{u\in\mathcal{S}}\sum_{i=0}^{\tau-1}f(u_{i})\left\vert
\mathcal{S}\text{ }(\tau,K,Q)\text{-spanning}\right.  \right\}  +0\\
&  \geq\underset{\tau\rightarrow\infty}{\overline{\lim}}\inf_{u\in\mathcal{S}%
}\frac{1}{\tau}\sum_{i=0}^{\tau-1}f(u_{i})\geq\inf_{u\in\mathcal{S}}\frac
{1}{\tau}\sum_{i=0}^{\tau-1}f(u_{i}).
\end{align*}
Since for $u\in\mathcal{S}$ there is $x\in K$ with $\pi\varphi(i,x,u)=0\in\pi
D$ for $i\in\{0,1,\dotsc,\tau-1\}$, the assertion for trivial unstable
subspace $E^{-}$ follows.

Now suppose that $E^{u}$ is nontrivial. We may assume that $P_{inv}%
(f,K,Q)<\infty$ and hence and all considered spanning sets are countable. Note
that by invariance of $E^{sc}$ and $E^{u}$ the induced system on $E^{u}$ is
well defined with trajectories $\pi\varphi(k,x,u),k\in\mathbb{N}$. For each
$u$ in a $(\tau,K,D)$-spanning set $\mathcal{S}$ define%
\[
\pi K_{u}:=\pi K\cap\bigcap_{t=0}^{\tau-1}\left(  \pi\varphi_{t,u}\right)
^{-1}(D).
\]
Thus $\pi K=%
{\textstyle\bigcup\nolimits_{u\in\mathcal{S}}}
\pi K_{u}$. Since $D$ is measurable, each set $\pi K_{u}$ is measurable as the
countable intersection of measurable sets. We denote the Lebesgue measure in
$\mathbb{R}^{d}$ by $\mu^{d}$ and the induced measure on $E^{u}$ by $\mu$. The
linear part of the affine-linear map $\pi\varphi_{\tau,u}(x)$ is given by
$(A^{+})^{\tau}$, hence it follows that
\[
\mu(\pi D)\geq\mu(\pi\varphi_{\tau,u}(\pi K_{u}))=\int\limits_{\pi
\varphi_{\tau,u}(\pi K_{u})}\mathrm{d}\mu=\int\limits_{\pi K_{u}}\left\vert
\det(A^{+})^{\tau}\right\vert \mathrm{d}\mu=\mu(\pi K_{u})\left\vert \det
A^{+}\right\vert ^{\tau}.
\]
Abbreviate $~\beta(\tau)=\inf_{(x,u)}(S_{\tau}f)(u)$, where the infimum is
taken over all $(\pi x,u)\in\pi K\times\mathcal{U}$ with $\pi\varphi
(i,x,u)\in\pi D$ for $i=0,\dotsc,\tau-1$. Then we find
\begin{align*}
e^{\beta(\tau)}\mu(\pi K) &  \leq\sum_{u\in\mathcal{S}}e^{(S_{\tau}f)(u)}%
\mu(\pi K_{u})\leq\sup_{u\in\mathcal{S}}\mu(\pi K_{u})\sum_{u\in\mathcal{S}%
}e^{(S_{\tau}f)(u)}\\
&  \leq\frac{\mu(\pi D)}{\left\vert \det A^{+}\right\vert ^{\tau}}\sum
_{u\in\mathcal{S}}e^{(S_{\tau}f)(u)}.
\end{align*}
Since this holds for every $(\tau,K,D)$-spanning set $\mathcal{S}$ and
$\mu^{d}(D)>0$ implies $\mu(\pi D)>0$, we find
\[
a_{\tau}(f,K,D)=\inf\{\sum_{u\in\mathcal{S}}e^{(S_{\tau}f)(u)}\left\vert
\mathcal{S}\text{ }(\tau,K,D)\text{-spanning}\right.  \}\geq\frac{\mu(\pi
K)}{\mu(\pi D)}e^{\beta(\tau)}\left\vert \det A^{+}\right\vert ^{\tau},
\]
implying%
\begin{align*}
&  P_{inv}(f,K,D)=\underset{\tau\rightarrow\infty}{\overline{\lim}}\frac
{1}{\tau}\log a_{\tau}(f,K,D)\geq\inf_{\tau}\frac{1}{\tau}\beta(\tau
)+\log\left\vert \det A^{+}\right\vert \\
&  =\inf_{(\tau,x,u)}\frac{1}{\tau}(S_{\tau}f)(u)+\log\left\vert \det
A^{+}\right\vert ,
\end{align*}
where the infimum is taken over all $(\tau,x,u)\in\mathbb{N}\times\pi
K\times\mathcal{U}$ with $\pi\varphi(i,x,u)\in\pi D$ for $i=0,\dotsc,\tau-1$.
\end{proof}

The next theorem is the main result of this paper. For linear discrete-time
control systems it provides a formula for the invariance pressure of control sets.

\begin{theorem}
\label{main}Consider a linear control system of the form (\ref{lin}) and
assume that the system without control restriction is controllable in
$\mathbb{R}^{d}$, the matrix $A$ is hyperbolic, and the control range $U$ is a
compact convex neighborhood of the origin with $U=\overline{\mathrm{int}U}$.
Let $D$ be the unique control set with nonvoid interior. Then $D$ is bounded
and for every compact set $K\subset D$ with nonvoid interior and every
potential $f\in C(U,\mathbb{R})$, the invariance pressure is given by
\[
P_{inv}(f,K,D)=\log\left\vert \det A^{+}\right\vert +\min_{u\in U}%
f(u)=h_{inv}(K,D)+\min_{u\in U}f(u).
\]

\end{theorem}

\begin{proof}
Theorems \ref{theorem_existence} and \ref{theorem_bounded} imply existence,
uniqueness, and boundedness of the control set $D$. Formula (\ref{h_inv})
implies that $h_{inv}(K,D)=\log\det A^{+}$ showing the second equality above.
Proposition \ref{prop_lower} and Corollary \ref{cor_upper} yield the bounds,%
\begin{equation}
\inf_{(\tau^{\prime},x^{\prime},u^{\prime})}\frac{1}{\tau^{\prime}}\sum
_{i=0}^{\tau^{\prime}-1}f(u_{i}^{\prime})\leq P_{inv}(f,K,Q)-\log\left\vert
\det A^{+}\right\vert \leq\inf_{(\tau,x,u)}\frac{1}{\tau}\sum_{i=0}^{\tau
-1}f(u_{i}),\label{M0}%
\end{equation}
where the first infimum is taken over all $(\tau^{\prime},x^{\prime}%
,u^{\prime})\in\mathbb{N}\times D\times\mathcal{U}$ with $\tau^{\prime}\geq d$
and $\pi\varphi(i,x^{\prime},u^{\prime})\in\pi D$ for $i\in\{0,\dotsc
,\tau^{\prime}-1\}$ and the second infimum is taken over all $\tau
\in\mathbb{N}$ with $\tau\geq d$ and all $\tau$-periodic controls $u$ with a
$\tau$-periodic trajectory $\varphi(\cdot,x,u)$ in $\mathrm{int}D$ such that
$u_{i}\in\mathrm{int}U$ for $i\in\{0,\dotsc,\tau-1\}$.

Note that there is a control value $u^{0}\in U$ with $f(u^{0})=\min_{u\in
U}f(u)$. Consider
\begin{equation}
f(u^{0})=\frac{1}{d}\sum_{i=0}^{d-1}f(u^{0})\leq\inf_{(\tau^{\prime}%
,x^{\prime},u^{\prime})}\frac{1}{\tau^{\prime}}\sum_{i=0}^{\tau^{\prime}%
-1}f(u_{i}^{\prime}), \label{M2}%
\end{equation}
where the infimum is taken over all triples $(\tau^{\prime},x^{\prime
},u^{\prime})\in\mathbb{N}\times K\times\mathcal{U}$ with $\tau^{\prime}\geq
d$ and $\pi\varphi(i,x^{\prime},u^{\prime})\in\pi D$ for $i\in\{0,\dotsc
,\tau^{\prime}-1\}$. Let $\varepsilon>0$. Then there is a control function
$u^{1}$ with values in a compact subset of $\mathrm{int}U$ such that%
\begin{equation}
\frac{1}{d}\sum_{i=0}^{d-1}f(u_{i}^{1})\leq\frac{1}{d}\sum_{i=0}^{d-1}%
f(u^{0})+\varepsilon. \label{M3}%
\end{equation}
By hyperbolicity of $A$ the matrix $I-A^{d}$ is invertible, and hence there
exists a unique solution $x^{1}$ of
\[
\left(  I-A^{d}\right)  x^{1}=\varphi(d,0,u^{1}).
\]
Now by linearity
\[
x^{1}=A^{d}x^{1}+\varphi(d,0,u^{1})=\varphi(d,x^{1},u^{1}).
\]
Since the values of $u^{1}$ are in $\mathrm{int}U$ and $(A,B)$ is
controllable, it follows that a neighborhood of $x^{1}$ can be reached in time
$d$ from $x^{1}$. Analogously, $x^{1}$ can be reached from every point in a
neighborhood of $x^{1}$ in time $d$. Hence in the intersection of these two
neighborhoods every point can be steered in time $2d$ into every other point.
This shows that $x^{1}$ is in the interior of the control set $D$, and the
corresponding trajectory $\varphi(i,x^{1},u^{1}),i\in\{0,\dotsc,d-1\}$,
remains by Proposition \ref{proposition_in} in the interior of $D$. Extending
$u^{1}$ to a $d$-periodic control again denoted by $u^{1}$ we find that the
control-trajectory pair $(u^{1}(\cdot),\varphi(\cdot,x^{1},u^{1}))$ is
$d$-periodic, the trajectory is contained in $\mathrm{int}D$ and all values
$u_{i}^{1}$ are in a compact subset of $\mathrm{int}U$. It follows that%
\begin{align*}
&  \inf_{(\tau^{\prime},x^{\prime},u^{\prime})}\frac{1}{\tau^{\prime}}%
\sum_{i=0}^{\tau^{\prime}-1}f(u_{i}^{\prime})\overset{(\ref{M2})}{\geq}%
f(u^{0})=\frac{1}{d}\sum_{i=0}^{d-1}f(u^{0})\overset{(\ref{M3})}{\geq}\frac
{1}{d}\sum_{i=0}^{d-1}f(u_{i}^{1})-\varepsilon\\
&  \geq\inf_{(\tau,x,u)}\frac{1}{\tau}\sum_{i=0}^{\tau-1}f(u_{i})-\varepsilon,
\end{align*}
where the first infimum is taken over all triples $(\tau^{\prime},x^{\prime
},u^{\prime})\in\mathbb{N}\times K\times\mathcal{U}$ with $\tau^{\prime}\geq
d$ and $\pi\varphi(i,x^{\prime},u^{\prime})\in\pi D$ for $i\in\{0,\dotsc
,\tau^{\prime}-1\}$ and the second infimum is taken over all $(\tau
,x,u)\in\mathbb{N}\times D\times\mathcal{U}$ such that the control-trajectory
pair $(u,\varphi(\cdot,x,u))$ is $\tau$-periodic with $\tau\geq d$, the
trajectory is contained in $\mathrm{int}D$, and the control values $u_{i}$ are
in a compact subset of $\mathrm{int}U$.

Using this in (\ref{M0}) we get
\[
\inf_{(\tau^{\prime},x^{\prime},u^{\prime})}\frac{1}{\tau^{\prime}}\sum
_{i=0}^{\tau^{\prime}-1}f(u_{i}^{\prime})\leq P_{inv}(f,K,Q)-\log\left\vert
\det A^{+}\right\vert \leq\inf_{(\tau^{\prime},x^{\prime},u^{\prime})}\frac
{1}{\tau^{\prime}}\sum_{i=0}^{\tau^{\prime}-1}f(u_{i}^{\prime})+\varepsilon.
\]
Since $\varepsilon>0$ is arbitrary, the assertion of the theorem follows.
\end{proof}

\begin{remark}
For partially hyperbolic control systems, Da Silva and Kawan prove in
\cite{KawaDS18} relations between invariance entropy and topological pressure
for the unstable determinant. In contrast to our framework, they consider the
topological pressure (with respect to the fibers) of associated random
dynamical systems obtained by endowing the space of controls with shift
invariant probability measures.
\end{remark}


\begin{thebibliography}{99}                                                                                               %


\bibitem {Cocosa1}F. Colonius, J.A.N. Cossich, and A. Santana, Invariance
pressure for control systems, J. Dyn. Diff. Equations 31(1) (2019), 1--23.

\bibitem {Cocosa2}F. Colonius, A. Santana, and J.A.N. Cossich, Invariance
pressure of control sets, SIAM J. Control Optim. 56(6) (2018), 4130-4147.

\bibitem {Cocosa3}F. Colonius, J.A.N. Cossich and A. Santana, Bounds for
invariance pressure, J. Diff. Equations 268) (2020), 7877-7896..

\bibitem {Colo18}F. Colonius, Invariance entropy, quasi-stationary measures
and control sets, Discrete and Continuous Dynamical Systems (DCDS-A) 38(4)
(2018), 2093-2123.

\bibitem {HiP20}D. Hinrichsen and A.J. Pritchard, Mathematical Systems Theory,
Vol. 2, Springer, 2021, in preparation.

\bibitem {KawaDS16}A. Da Silva and C. Kawan, Invariance entropy of hyperbolic
control sets, Discrete and Continuous Dynamical Systems (DCDS-A) 36(1) (2016), 97-136.

\bibitem {KawaDS19}A. Da Silva and C. Kawan, Lyapunov exponents and partial
hyperbolicity of chain control sets on flag manifolds, Israel Journal of
Mathematics 232 (2019), 947-1000.

\bibitem {HuanZ18}Y. Huang and X. Zhong, Carath\'{e}odory--Pesin structures
associated with control systems, Systems and Control Letters 112 (2018), pp. 36-41.

\bibitem {Kawa11b}C. Kawan, Invariance entropy of control sets, SIAM J.
Control Optim. 49 (2011), 732-751.

\bibitem {Kawa13}C. Kawan, Invariance Entropy for Deterministic Control
Systems. An Introduction. LNM Vol. 2089, Springer, Berlin, 2013.

\bibitem {KawaDS18}C. Kawan and A. Da Silva, Invariance entropy for a class of
partially hyperbolic sets, Mathematics of Control, Signals and Systems 30(18)
(2018), https://doi.org/10.1007/s00498-018-0224-2.

\bibitem {NEMM04}G. Nair, R. J. Evans, I. Mareels, and W. Moran, Topological
feedback entropy and nonlinear stabilization, IEEE Trans. Aut. Control 49
(2004), 1585--1597.

\bibitem {PatSM07}M. Patr\~{a}o and L. San Martin, Semiflows on topological
spaces: Chain transitivity and semigroups, J. Dyn. Diff. Equations, 19 (2007), 155--180.

\bibitem {Son98}E. Sontag, Mathematical Control Theory. Deterministic Finite
Dimensional Systems, 2nd ed., Springer-Verlag, New York 1998.

\bibitem {WangHS19}Tao Wang, Yu Huang, and Hai-Wei Sun, Measure-theoretic
invariance entropy for control systems, SIAM J. Control Optim. 57(1) (2019), 310-333.

\bibitem {WinD63}J. Wing and C. A. Desoer, The multiple-input minimal-time
regulator problem (general theory), IEEE\ Trans. Automatic Control AC-8(2)
(1963), 125-136.

\bibitem {Wirth98}F. Wirth, Dynamics and controllability of nonlinear
discrete-time control systems, IFAC Proceedings Volumes 31 (1998), 267-272.

\bibitem {ZHuag19}X. Zhong, Y. Huang, Invariance pressure dimensions for
control systems, J. Dyn. Diff. Equations 31 (2019), 2205-2222.
\end{thebibliography}
\end{document}